\documentclass[pdf]{article}

\usepackage{pdfpages}
\usepackage{hyperref}
\usepackage{xcolor}
\usepackage[utf8]{inputenc}
\usepackage{amsmath}
\usepackage{amsfonts}
\usepackage{amssymb}
\usepackage{amsthm}
\usepackage{marginnote}
\usepackage{enumerate}
\usepackage{graphicx}
\usepackage[normalem]{ulem}
\usepackage{fancyhdr}
\newcommand{\defstyle}[1]{\textbf{#1}}
\newcommand{\myprob}[1]{\mathbb P \left[ #1 \right]}

\newcommand{\norm}[1]{\left| #1 \right|}

\newcommand{\tv}[2]{||#1-#2||}

\newcommand{\bs}[1]{\boldsymbol{#1}}

\newcommand{\del}[1]{}

\newcommand{\mar}[1]{\marginpar{\scriptsize  #1}}
\newcommand{\unwritten}[1]{}

\newcommand{\pmm}{PMM}
\newcommand{\pms}{PM}

\newcommand{\supp}{\mathrm{supp}}
\newcommand{\diam}{\mathrm{diam}}
\newcommand{\distortion}{\mathrm{dis}}
\newcommand{\oball}[2]{B_{#1}(#2)}
\newcommand{\cball}[2]{\overline{B}_{#1}(#2)}
\newcommand{\pcball}[2]{\overline{#1}^{(#2)}}
\newcommand{\poball}[2]{{#1}^{(#2)}}
\newcommand{\cnei}[2]{{N}_{#1}(#2)}
\newcommand{\hp}{d_{HP}}
\newcommand{\ghp}{d_{GHP}}
\newcommand{\cghp}{d^c_{GHP}}
\newcommand{\gh}{d_{GH}}
\newcommand{\cgh}{d^c_{GH}}
\newcommand{\hausdorff}{d_H}
\newcommand{\prokhorov}{d_P}

\newcommand{\restrict}[2]{{%
		\left.\kern-\nulldelimiterspace %
		#1 %
		\vphantom{\big|} %
		\right|_{#2} %
}}
\newcommand{\rooot}{root}

\newcommand{\mstar}{\mathfrak N_*}

\newcommand{\mcstar}{\mathfrak N^c_*}
\newcommand{\mcstarzero}{\mathfrak N^c}
\newcommand{\mmstar}{\mathfrak M_*}
\newcommand{\mmcstar}{\mathfrak M^c_*}

\newcommand{\mwstar}[1]{\mathfrak M_{#1}}

\newcommand{\dirac}[1]{\delta_{#1}}

\theoremstyle{theorem}
\newtheorem{theorem}{Theorem}[section]
\newtheorem{lemma}[theorem]{Lemma}
\newtheorem{proposition}[theorem]{Proposition}
\newtheorem{corollary}[theorem]{Corollary}

\theoremstyle{definition}
\newtheorem{definition}[theorem]{Definition}

\theoremstyle{definition}
\newtheorem{remark}[theorem]{Remark}
\newtheorem{convention}[theorem]{Convention}

\theoremstyle{theorem}

\numberwithin{equation}{section}

\makeatletter
\let\orgdescriptionlabel\descriptionlabel
\renewcommand*{\descriptionlabel}[1]{%
	\let\orglabel\label
	\let\label\@gobble
	\phantomsection
	\edef\@currentlabel{#1}%
	\let\label\orglabel
	\orgdescriptionlabel{#1}%
}
\makeatother

\begin{document}

\title{Metrization of the Gromov-Hausdorff (-Prokhorov) Topology for Boundedly-Compact Metric Spaces}

\author{Ali Khezeli\footnote{Tarbiat Modares University, khezeli@modares.ac.ir}}

\maketitle

\begin{abstract}
	In this work, a metric is presented on the set of boundedly-compact pointed metric spaces that generates the Gromov-Hausdorff topology. A similar metric is defined for measured metric spaces that generates the Gromov-Hausdorff-Prokhorov topology. This extends previous works which consider only length spaces or discrete metric spaces. Completeness and separability are also proved for these metrics. Hence, they provide the measure theoretic requirements to study random (measured) boundedly-compact pointed metric spaces, which is the main motivation of this work.
	In addition, we present a generalization of the classical theorem of Strassen which is of independent interest. This generalization proves an equivalent formulation of the Prokhorov distance of two finite measures, having possibly different total masses, in term of approximate coupling. A Strassen-type result is also proved for the Gromov-Hausdorff-Prokhorov metric for compact spaces.
\end{abstract}

\tableofcontents

\section{Introduction}

Subsection~\ref{subsec:intro-Gromov} below provides an introduction to the notion of Gromov-Hausdorff convergence on the set of all \textit{boundedly-compact pointed metric spaces}. The contributions of the present paper are introduced in Subsection~\ref{subsec:intro-thiswork}.

\subsection{Introduction to the Gromov-Hausdorff Topology}
\label{subsec:intro-Gromov}
\textbf{The Gromov-Hausdorff Metric.} The Hausdorff metric, denoted by $\hausdorff$, defines the distance of two compact subsets of a given metric space. Gromov defined
a metric on the set $\mcstarzero$ of all %
compact metric spaces which are not necessarily contained in a given space (isometric metric spaces are regarded equivalent). This metric is called the \textit{Gromov-Hausdorff metric} in the literature. %
The distance of two compact metric spaces $X$ and $Y$ is defined by
\begin{equation}
\label{eq:GH}
\cgh(X,Y):=\inf \hausdorff(f(X),g(Y)),
\end{equation}
where the infimum is over all metric spaces $Z$ and all pairs of isometric embeddings $f:X\to Z$ and $g:Y\to Z$ (an isometric embedding is a distance-preserving map which is not necessarily surjective).

The Gromov-Hausdorff metric has been defined for group-theoretic purposes. However, it has found important applications in probability theory as well, 
since it enables one to study \textit{random compact metric spaces}. Specially, this is used in the study of \textit{scaling limits} of random graphs and other random objects. This goes back to the novel work of Aldous \cite{Al91-crtI} who proved that a random tree with $n$ vertices, chosen uniformly at random and scaled properly, converges to a random object called \textit{the Brownian continuum random tree} in a suitable sense as $n$ tends to infinity. Using Gromov's definition, Aldous's result can be restated in terms of weak convergence of probability measures on $\mcstarzero$ (see~\cite{Le06} and~\cite{EvPiWi06}). Since then, scaling limits of various random discrete models have also been studied. 
An important topological property needed for probability-theoretic applications is that the set $\mcstarzero$ (or other relevant sets) %
is complete and separable, and hence, can be used as a standard probability space.

\textbf{The Gromov-Hausdorff-Prokhorov Metric.}
The \textit{Prokhorov metric}, denoted by $\prokhorov$, defines the distance of two finite measures on a common metric space. By using this metric, the Gromov-Hausdorff metric is generalized to define the distance of two  \textit{compact measured metric spaces} (\cite{GrPfWi09}, \cite{Mi09}, \cite{bookVi10} and~\cite{AbDeHo13}),
where a compact measured metric space is a compact metric space $X$ together with a finite measure $\mu$ on $X$.
The metric is usually called the \textit{Gromov-Hausdorff-Prokhorov} metric and is defined by 
\[
\cghp((X,\mu),(Y,\nu)):=\inf \{\hausdorff(f(X),g(Y))\vee \prokhorov(f_*\mu,g_*\nu)\}, 
\]
where the infimum is over all metric spaces $Z$ and isometric embeddings $f:X\to Z$ and $g:Y\to Z$ ($f_*\mu$ denotes the push-forward of the measure $\mu$ by $f$). %

This metric plays an important role in defining and studying random measured metric spaces (see e.g., \cite{AbDeHo13} and the papers citing it).
In particular, since every discrete set can be naturally equipped with the counting measure, this metric can be used to prove stronger convergence results in scaling limits of discrete objects. Also, it has applications in mass-transportation problems (see e.g., \cite{bookVi10} and the references mentioned therein). %

\textbf{The Non-Compact Case.}
To relax the assumption of compactness,
it is convenient to consider \textit{boundedly-compact} metric spaces; i.e., metric spaces in which every closed ball of finite radius is compact. Also, it is important in many aspects to consider \textit{pointed} metric spaces; i.e., metric spaces with a distinguished point, which is called \textit{the origin} here. %
Then, the notion of \textit{Gromov-Hausdorff convergence} is defined for sequences of boundedly-compact pointed metric spaces (see e.g.,~\cite{bookBBI}), which goes back to Gromov~\cite{bookGr99}. 
Heuristically, the idea is to consider large balls centered at the origins and compare them using the Gromov-Hausdorff metric in the compact case (the precise definition takes into account the discontinuity issues caused by the points which are close to the boundaries of the balls). %
This gives a topology on the set $\mstar$ of boundedly-compact pointed metric spaces, called the \textit{Gromov-Hausdorff topology}. 
The notion of \textit{Gromov-Hausdorff-Prokhorov} convergence and topology~\cite{bookVi10} (also called \textit{measured Gromov-Hausdorff convergence}) is  defined similarly on the set $\mmstar$ of boundedly-compact pointed measured metric spaces (in which the measures are \textit{boundedly-finite}). %
The next subsection provides more discussion on the matter.

\subsection{Introduction to the Contributions of the Present Paper}
\label{subsec:intro-thiswork}
The main focus of this work is on boundedly-compact pointed metric spaces and measured metric spaces. In the boundedly-compact case, under some restrictions on the metric spaces under study, similar metrics are defined in the literature that generate the Gromov-Hausdorff (-Prokhorov) topology restricted to the corresponding subsets of $\mstar$ or $\mmstar$. For instance, \cite{AbDeHo13} considers only \textit{length spaces} (i.e., metric spaces in which the distance of any two points is the infimum length of the curves connecting them) and~\cite{dimensionI} considers discrete metric spaces. Also, in the case of graphs (where every graph is equipped with the graph-distance metric), the Benjamini-Schramm metric~\cite{BeSc01} does the job. %
These papers use the corresponding metrics to study random \textit{real trees}, random discrete metric spaces and random graphs respectively, in the non-compact case.

{The main contribution of the present paper is} the definition of
a metric on the set $\mstar$ of \textit{all} boundedly-compact pointed metric spaces (which are not necessarily length spaces or discrete spaces) that generates the Gromov-Hausdorff topology. The same is done for measured metric spaces as well 
(connections with the metric defined in~\cite{AtLoWi16} will be discussed in the last section). This enables one to define and study random (measured) boundedly-compact pointed metric spaces, which is the main motivation of this paper.

To define the distance of two boundedly-compact pointed metric spaces $(X,o)$ and $(X',o')$, the idea is, as in the Gromov-Hausdorff convergence, to compare large balls centered at $o$ and $o'$ (this idea is sometimes called \textit{the localization method}, which is commonly used in various situations in the literature some of which are discussed in Section~\ref{sec:special}). 
There are some pitfalls caused by boundary-effects of the balls; e.g., the value $\cgh(\cball{r}{o},\cball{r}{o'})$, where $\cball{r}{o}$ is the closed ball of radius $r$ centered at $o$, is not monotone in $r$. The definition of this paper is based on the following value, which has a useful monotonicity property: %
$\inf\left\{\cgh(\cball{r}{o},Y) \right\}$,
where the infimum is over all compact subsets $Y\subseteq X'$ such that $Y\supseteq \cball{r-1/r}{o'}$ (the last condition can also be removed and most of the results remain valid). Here, a version of the metric $\cgh$ for pointed metric spaces should be used. 
For measured metric spaces, a similar metric is also provided which gives the Gromov-Hausdorff-Prokhorov topology. The definition of this metric is based on a similar idea.

It is also proved that the set $\mstar$ (resp. $\mmstar$) of boundedly-compact pointed (measured) metric spaces is complete and separable, and hence, can be used as a standard probability space.  
This is important if one wants to consider random (measured) metric spaces in the boundedly-compact %
case. 

Meanwhile, as a tool in the proofs, a generalization of K\"onig's's infinity lemma is proved for compact sets, which is of independent interest. 
The arguments based on this lemma are significantly simpler in comparison with similar arguments in the literature. %

Other variants of the metric are also available, for instance 
\begin{equation*}
\int_0^{\infty} e^{-r}\left(1\wedge \cgh(\cball{r}{o}, \cball{r}{o'}) \right)dr.
\end{equation*}
By the results of this paper, one can show that this formula defines a metric on $\mstar$ as well and has similar properties (formulas like this are common in various settings in the literature; e.g., \cite{AbDeHo13}),
but the definition of the present paper enables one to have more quantitative bounds in the arguments.

In addition, a generalization of Strassen's theorem~\cite{St65} is presented, which is of independent interest and is useful in the arguments. The result provides an equivalent formulation of the Prokhorov distance between two given finite measures on a common metric space. The original theorem of Strassen does this in the case of probability measures. A Strassen-type result is also presented for the Gromov-Hausdorff-Prokhorov metric in the compact case.

\del{The second main contribution of the paper is a generalization of the Gromov-Hausdorff metric to the set of metric spaces equipped with some additional %
structure. 
This unifies most of the examples mentioned in the previous subsection, where the additional structure is a point of $X$, a tuple of measures on $X$, etc.
In general, for a compact metric space $X$, let $\tau(X)$ be a set that represents  the set of possible additional structures on $X$.
Under some assumptions on $\tau$, which are tried to be minimalistic, a version of the Gromov-Hausdorff metric is defined and its completeness and separability is proved.
Moreover, a metric is defined in the boundedly-compact pointed case under further assumptions.
Some new specific examples of additional structures are also studied. This generalization might also be useful in the future if one wants to study a new type of additional structures on compact (or boundedly-compact) metric spaces. }

Finally, the connections to other notions in the literature are discussed. 
This includes random measures, Benjamini-Schramm metric for graphs, the Skorokhod space of c\`adl\`ag functions, the work of~\cite{AtLoWi16} for \textit{metric measure spaces}, and more.

The structure of the paper is as follows. Section~\ref{sec:HausdorffProkhorov} recalls the Hausdorff and Prokhorov metrics and also provides the generalization of Strassen's theorem. In Section~\ref{sec:ghp}, the Gromov-Hausdorff-Prokhorov metric is recalled in the compact case and a Strassen-type theorem is proved for it. The metric is also extended to the general boundedly-compact pointed case (it contains the Gromov-Hausdorff metric as a special case). The properties of this metric are also studied therein. 
Finally, Section~\ref{sec:special} discusses special cases of the metric which already exist in the literature and also discusses the connections to other notions.

\section{The Hausdorff and Prokhorov Metrics}
\label{sec:HausdorffProkhorov}

In this section, the definitions and basic properties of the Hausdorff and Prokhorov metrics are recalled. Also, a generalization of Strassen's  theorem~\cite{St65} is provided (Theorem~\ref{thm:strassen}) which gives an equivalent formulation of the Prokhorov metric. It will be used in the next section.

\subsection{Notations}

The set of nonnegative real numbers 
is denoted by $\mathbb R^{\geq 0}$. The minimum and maximum binary operators 
are denoted by $\wedge$ and $\vee$ respectively. 

For all metric spaces $X$ in this paper, the metric on $X$ is always denoted by $d$ if there is no ambiguity. 
For a closed subset $A\subseteq X$, the (closed) $r$-neighborhood of $A$ in $X$ is the set $\cnei{r}{A}:=\{x\in X:\exists y\in A: d(x,y)\leq r \}$.
The complement of  $A$ is denoted by $A^c$ or $X\setminus A$.
The two projections from $X\times Y$ onto $X$ and $Y$ are denoted by $\pi_1$ and $\pi_2$ respectively.
Also, all measures on $X$ are assume to be Borel measures. The Dirac measure at $a\in X$ is denoted by $\dirac{a}$.  If $\mu$ is a measure on $X$, the \textit{total mass} of $\mu$ is defined by
\[
||\mu||:=\mu(X).
\] 
If in addition, $\rho:X\rightarrow Y$ is measurable, $\rho_*\mu$ denotes the push-forward of $\mu$ under $\rho$; i.e., $\rho_*\mu(\cdot) = \mu(\rho^{-1}(\cdot))$. If $\mu$ and $\nu$ are measures on $X$, the \textit{total variation distance} of $\mu$ and $\nu$ is defined by 
\[
\tv{\mu}{\nu}:= \sup\{\norm{\mu(A)-\nu(A)}: A\subseteq X \}.
\]

\subsection{The Hausdorff Metric}
The following definitions and results are borrowed from~\cite{bookBBI}.
Let $Z$ be a metric space. %
For two closed subsets $A,B\subseteq Z$, the \defstyle{Hausdorff distance} of $A$ and $B$ is defined by
\begin{equation}
\label{eq:hausdorff}
\hausdorff(A,B):= \inf\{\epsilon\geq 0: A\subseteq \cnei{\epsilon}{B} \text{ and } B\subseteq \cnei{\epsilon}{A} \}.
\end{equation}
Let $\mathcal F(Z)$ be the set of closed subsets of $Z$. It is well known that $\hausdorff$ is a metric on $\mathcal F(Z)$. Also, if $Z$ is complete and separable, then $\mathcal F(Z)$ is also complete and separable. In addition, if $Z$ is compact, then $\mathcal F(Z)$ is also compact. See e.g., Proposition~7.3.7 and Theorem~7.3.8 of~\cite{bookBBI}.

\subsection{The Prokhorov Metric}

Fix a complete separable metric space $Z$. For two finite Borel measures $\mu$ and $\nu$ on $Z$, the \defstyle{Prokhorov distance} of $\mu$ and $\nu$ (see e.g., \cite{bookKa17randommeasures}) is defined by 
\begin{eqnarray}
\label{eq:dProkhorov}
\prokhorov(\mu,\nu):=\inf\{\epsilon>0: \forall A: \mu(A)\leq \nu(\cnei{\epsilon}{A})+\epsilon, \nu(A)\leq \mu(\cnei{\epsilon}{A})+\epsilon\},
\end{eqnarray}
where $A$ ranges over all closed subsets of $Z$.

It is well known that $\prokhorov$ is a metric on the set of finite Borel measures on $Z$ and makes it a complete and separable metric space. Moreover, the topology generated by this metric coincides with that of weak convergence (see e.g., \cite{bookKa17randommeasures}).

The following theorem is the main result of this subsection. It provides another formulation of the Prokhorov distance using the notion of \textit{approximate couplings}~\cite{AdBrGoMi17} and will be useful afterwards. %
Let $\alpha$ be a finite Borel measure on $X\times X$. The \defstyle{discrepancy} of $\alpha$ w.r.t. $\mu$ and $\nu$~\cite{AdBrGoMi17} is defined by
\[
D(\alpha;\mu,\nu):= \tv{\pi_{1*}\alpha}{\mu} + \tv{\pi_{2*}\alpha}{\nu}.
\]
One has $D(\alpha;\mu,\nu)=0$ if and only if $\alpha$ is a coupling of $\mu$ and $\nu$; i.e., $\pi_{1*}\alpha=\mu$ and $\pi_{2*}\alpha=\nu$.

\begin{theorem}[Generalized Strassen's Theorem]
	\label{thm:strassen}
	Let $\mu$ and $\nu$ be finite Borel measures on a complete separable metric space $Z$. %
	\begin{enumerate}[(i)]
		\item \label{thm:Prokhorov:1} $\prokhorov(\mu,\nu)\leq \epsilon$ if and only if there is a Borel measure $\alpha$ on $Z\times Z$ such that
		\begin{equation}
		\label{eq:thm:Prokhorov:0}		
		D(\alpha;\mu,\nu)+ \alpha(\{(x,y): d(x,y)>\epsilon \}) \leq \epsilon.
		\end{equation}

		\item Equivalently,
		\begin{equation}
		\label{eq:thm:Prokhorov:1}
		\prokhorov(\mu,\nu)=\min \{\epsilon\geq 0: \exists \alpha: D(\alpha;\mu,\nu)+ \alpha(\{(x,y): d(x,y)>\epsilon \}) \leq \epsilon \}
		\end{equation}
		and the minimum is attained.
		
		\item \label{thm:Prokhorov:2} In addition, if $\mu(Z)\leq\nu(Z)$, then the infimum in~\eqref{eq:thm:Prokhorov:1} is attained for $\epsilon:=\prokhorov(\mu,\nu)$ and some $\alpha$ such that $\pi_{1*}\alpha=\mu$ and $\pi_{2*}\alpha\leq \nu$. Moreover, $\alpha$ can be chosen to be supported on $\supp(\mu)\times \supp(\nu)$.
	\end{enumerate}
\end{theorem}

\begin{proof}%
	Let $\epsilon\geq 0$ and $\alpha$ be a measure 
	satisfying~\eqref{eq:thm:Prokhorov:0}.
	We will prove that $\prokhorov(\mu,\nu)\leq\epsilon$. Let $\epsilon_1:= \tv{\pi_{1*}\alpha}{\mu}$, $\epsilon_2:=\tv{\pi_{2*}\alpha}{\nu}$ and $\delta:= \alpha(\{(x,y): d(x,y)>\epsilon \})$. Let $A\subseteq Z$ be a closed subset and $B:=\{(x,y): x\in A, d(x,y)\leq\epsilon \}$. One has $\pi_2(B)=\cnei{\epsilon}{A}$. Therefore,
	\begin{eqnarray*}
		\mu(A) &\leq& \pi_{1*}\alpha(A)+\epsilon_1 = \alpha(\pi_1^{-1}(A))+\epsilon_1\\
		&\leq & \alpha(B)+\epsilon_1+\delta \leq \alpha(\pi_2^{-1}(\cnei{\epsilon}{A}))+\epsilon_1+\delta \\
		&=& \pi_{2*}\alpha(\cnei{\epsilon}{A}))+\epsilon_1+\delta 
		\leq  \nu(\cnei{\epsilon}{A}) + \epsilon_1+\epsilon_2+\delta\\
		&\leq& \nu(\cnei{\epsilon}{A}) + \epsilon,
	\end{eqnarray*}
	where the last inequality holds by the assumption~\eqref{eq:thm:Prokhorov:0}.
	Similarly, one can show $\nu(A)\leq \mu(\cnei{\epsilon}{A})+\epsilon$. Since this holds for all $A$, one gets $\prokhorov(\mu,\nu)\leq\epsilon$.

	Conversely, assume $\prokhorov(\mu,\nu)\leq\epsilon$. One can assume $\nu(Z)=\mu(Z)+\delta$ and $\delta\geq 0$ without loss of generality. Let $r>\epsilon$ be arbitrary. The former assumption implies that $\nu(A)\leq \mu(\cnei{r}{A})+r$ for every closed set $A\subseteq Z$. It follows that 
	\begin{equation}
	\label{eq:thm:Prokhorov:2}
	\nu(Z\setminus A)\geq \mu(Z\setminus \cnei{r}{A}) -r +\delta.
	\end{equation}
	Let $B\subseteq Z$ be an arbitrary closed subset, $s>r$ be arbitrary and $A$ be the closure of $Z\setminus \cnei{s}{B}$. %
	Note that $Z\setminus A\subseteq \cnei{s}{B}$ and $Z\setminus \cnei{r}{A}\supseteq B$. It follows from~\eqref{eq:thm:Prokhorov:2} that $\nu(\cnei{s}{B}) \geq \mu(B)-r+\delta$. By letting $r$ and $s$ tend to $\epsilon$ and by $\cap_{s>\epsilon}{\cnei{s}{B}}=B^{\epsilon}$, one gets that
	\[\mu(B)\leq \nu(\cnei{\epsilon}{B}) + \epsilon-\delta,\] 
	for all closed sets $B\subseteq Z$.
	Now, add a point $a$ to $Z$, let $Z'=Z\cup\{a\}$ and let 
	$\nu':=\nu+(\epsilon-\delta)\dirac{a}$, which is a measure on $Z'$.
	Let $K:=\{(x,y)\in Z\times Z:  d(x,y)\leq\epsilon\} \cup (Z\times \{a\})$. Then, for any closed subset $A\subset Z$, one has $\mu(A)\leq \nu'(\{y\in Z': \exists x\in A: (x,y)\in K \})$. Therefore, by Lemma~\ref{lem:Hall} below, one finds a measure $\beta$ on $K$ such that $\pi_{1*}\beta=\mu$ and $\pi_{2*}\beta\leq \nu'$. Let $\gamma$ be the restriction of $\beta$ to $Z\times Z$. One has $\pi_{1*}\gamma\leq \mu$ and $\pi_{2*}\gamma\leq \nu$. Let $\mu_1:=\mu- \pi_{1*}\gamma$ and $\nu_1:=\nu-\pi_{2*}\gamma$. 
	The assumption $\mu(Z)\leq \nu(Z)$ implies that $\mu_1(Z)\leq \mu_1(Z)$. Therefore, if $\nu_1=0$, then $\mu_1=0$ and $\gamma$ has the desired properties. So, assume $\nu_1\neq 0$. Also, one can obtain $||\mu_1|| \leq \nu'(a) = \epsilon - \delta$. Define
	\begin{equation}
	\label{eq:thm:Prokhorov:3}
	\alpha:=\gamma + \frac 1{\nu_1(Z)} \mu_1 \otimes \nu_1.
	\end{equation}
	We claim that $\alpha$ satisfies the desired properties.
	It is straightforward that $\pi_{1*}\alpha = \mu$ and $\pi_{2*}\alpha\leq \nu$.
	This implies that $D(\alpha;\mu,\nu) = ||\nu||-||\pi_{2*}\alpha|| = ||\nu|| - ||\mu|| = \delta$. Also, since $\gamma$ is supported on $K$, \eqref{eq:thm:Prokhorov:3} implies that 
	\[
	\alpha(K^c)\leq ||\frac 1{\nu_1(Z)} \mu_1 \otimes \nu_1|| = ||\mu_1|| \leq \epsilon-\delta.
	\]
	Therefore, $D(\alpha;\mu,\nu)+ \alpha(K^c) \leq \epsilon$. So, $\alpha$ satisfies~\eqref{eq:thm:Prokhorov:0}. Finally, it can be seen that $\alpha$ is supported on $\supp(\mu)\times \supp(\nu)$ and the claim is proved.
\end{proof}

It is shown below how Theorem~\ref{thm:strassen} implies Strassen's theorem~\cite{St65}.

\begin{corollary}[Strassen's Theorem]
	\label{cor:strassen}
	Let $\mu$ and $\nu$ be finite Borel measures on $Z$ %
	such that $\mu(Z)=\nu(Z)$. Then, there exists a coupling $\alpha$ of $\mu$ and $\nu$ such that
	\begin{equation}
	\label{eq:strassen}
	\alpha(\{(x,y): d(x,y)>\epsilon \}) \leq \epsilon,
	\end{equation}
	where $\epsilon:=\prokhorov(\mu,\nu)$.
\end{corollary}
\begin{proof}
	Let $\alpha$ be the measure in part~\eqref{thm:Prokhorov:2} of Theorem~\ref{thm:strassen} for $\epsilon:=\prokhorov(\mu,\nu)$. One has $\pi_{1*}\alpha = \mu$ and $\pi_{2*}\alpha\leq \nu$. The assumption $\mu(Z)=\nu(Z)$ implies that $\pi_{2*}\alpha=\nu$. So, $\alpha$ is a coupling of $\mu$ and $\nu$ and $D(\alpha;\mu,\nu)=0$. Since the infimum in~\eqref{eq:thm:Prokhorov:1} is attained at $\alpha$, one has $\alpha(\{(x,y): d(x,y)>\epsilon \}) \leq \epsilon$ and the claim is proved.
\end{proof}

\begin{remark}
	A variant of the Prokhorov metric is defined in~\cite{AdBrGoMi17} by a formula similar to~\eqref{eq:thm:Prokhorov:1} (by changing the $+$  to $\vee$ in~\eqref{eq:thm:Prokhorov:1}). This definition, although not identical to the classical Prokhorov metric~\eqref{eq:dProkhorov}, 
	only differs by a factor at most 2, and hence, generates the same topology.
\end{remark}

The following lemma is used in the proof of Theorem~\ref{thm:strassen}. It is a continuum version of Hall's marriage theorem and also generalizes Theorem~11.6.3 of~\cite{bookDu02}.
\begin{lemma}
	\label{lem:Hall}
	Let $X$ and $Y$ be separable metric spaces and $\mu$ and $\nu$ be finite Borel measures on $X$ and $Y$ respectively. Assume $K\subseteq X\times Y$ is a closed subset such that for every closed set $A\subseteq X$, one has $\mu(A)\leq \nu(K(A))$, where $K(A):=\{y\in Y: \exists x\in A: (x,y)\in K \}$. Then there is a Borel measure $\alpha$ on $K$ such that $\pi_{1*}\alpha=\mu$ and $\pi_{2*}\alpha\leq \nu$.
\end{lemma}

\begin{proof}
	If $\mu$ and $\nu$ have finite supports and integer values, then the claim follows easily from Hall's marriage theorem (to show this, by splitting the atoms of $\mu$ and $\nu$ into finitely many points, one can reduce the problem to the case where every atom has measure one). 
	By scaling, the same holds if $\mu$ and $\nu$ have finite supports and rational values. 
	Note that such measures are dense in the set of finite measures (see e.g., Lemma~4.5 in~\cite{bookKa17randommeasures}).  %
	
	Now, let $\mu$ and $\nu$ be arbitrary measures that satisfy the assumptions of the lemma. By the above arguments, there exist sequences $(\mu_n)_n$ and $(\nu_n)_n$ of finite measures on $X$ and $Y$ respectively that converge weakly to $\mu$ and $\nu$ respectively and every $\mu_n$ or $\nu_n$ has finite support and rational values. So the claim holds for $\mu_n$ and $\mu_n$ for each $n$. 
	For $m\in\mathbb N$, one can find $n=n(m)$ such that $\prokhorov(\mu_n,\mu)<\frac 1 m$ and $\prokhorov(\nu_n,\nu)<\frac 1 m$. Add a point $a$ to $Y$ and define $\nu'_n:=\nu_n+ \frac 2 m \delta_a$ and 
	\[
	K_{m}:= \{(x,y): \exists (x',y')\in K: d(x,x')\leq \frac 1 m, d(y,y')\leq \frac 1 m \} \cup (X\times \{a\}).
	\]
	Therefore, for any closed set $A\subseteq X$, one has
	\begin{eqnarray*}
		\mu_n(A) &\leq &\mu(\cnei{1/m}{A}) + \frac 1 m
		\leq \nu(K(\cnei{1/m}{A})) + \frac 1 m\\
		&\leq& \nu_n
		(\cnei{1/m}{K(\cnei{1/m}{A})})
		+\frac 2 m = \nu'_n(K_{m}(A)),
	\end{eqnarray*}
	where $K_{m}(A)\subseteq Y\cup\{a\}$ is defined similarly to $K(A)$. 
	Note that $\mu_n$ and $\nu'_n$ have finite supports and rational values. So the claim of the lemma holds for them. Therefore, 
	one can find a Borel measure $\alpha_{m}$ on $K_{m}$ such that $\pi_{1*}\alpha_{m}=\mu_n$ and $\pi_{2*}\alpha_{m}\leq \nu'_n$. By the finiteness of $\mu$ and $\nu$, it is easy to see that the set of measures $\alpha_{m}$ is tight. So one finds a convergent subsequence of $\alpha_{m}$'s, say converging weakly to $\alpha$. Since the sets $K_{m}$ are closed and nested, it can be seen that $\alpha$ is supported on $K_{m}$ for any $m$, and hence, it is supported on $\cap_{m} K_{m} = K\cup (X\times\{a\})$. Moreover, since $\alpha_{m}(X\times\{a\})\leq \frac 2 m$, $X\times\{a\}$ is disjoint from $K$ and $K$ is closed, it follows that $\alpha$ is supported on $K$ only. Finally, by $\pi_{1*}\alpha_{m} = \mu_n$ and $\pi_{2*}\alpha_{m}\leq \nu_n+\frac 2 m \delta_a$, one can get $\pi_{1*}\alpha = \mu$ and $\pi_{2*}\alpha\leq \nu$. So, the claim is proved.
\end{proof}

\section{The Gromov-Hausdorff-Prokhorov Metric}
\label{sec:ghp}

This section presents the main contribution of the paper. Roughly speaking, the Gromov-Hausdorff metric and the Gromov-Hausdorff-Prokhorov metric are generalized to the non-compact case (Subsection~\ref{subsec:ghp-nonCompact}); and more precisely, to boundedly-compact pointed (measured) metric spaces. Here, no further restrictions on the metric spaces are needed (e.g., being a length space or a discrete space as in~\cite{AbDeHo13} and~\cite{dimensionI} respectively). As mentioned in the introduction, this provides a metrization of the Gromov-Hausdorff (-Prokhorov) topology, where the latter has been defined earlier in the literature. In addition, completeness, separability, pre-compactness and weak convergence of probability measures are studied for the Gromov-Hausdorff (-Prokhorov) metric. Moreover, in the compact case, a Strassen-type theorem is proved for the Gromov-Hausdorff-Prokhorov metric.

Since the Gromov-Hausdorff metric is a special case of the Gromov-Hausdorff-Prokhorov metric (by considering metric spaces equipped with the zero measure), only the latter is discussed in this section. If the reader is interested in the Gromov-Hausdorff metric only, he or she can assume that all of the measures in this section are equal to zero (except in Subsection~\ref{subsec:weak}). 
Further discussion is provided in Subsection~\ref{subsec:GH}.

\subsection{Pointed Measured Metric (\pmm{}) Spaces}
\label{subsec:pmmSpace}

This subsection provides the basic definitions and properties regarding (measured) metric spaces.
Given metric spaces $X$ and $Z$, a function $f:X\to Z$ is an \textbf{isometric embedding} if it preserves the metric; i.e., $d(f(x_1),f(x_2))=d(x_1,x_2)$ for all $x_1,x_2\in X$. It is an \textbf{isometry} if it is a surjective isometric embedding. 
For a metric space $X$, $x\in X$ and $r\geq 0$, let
\begin{eqnarray*}
	\oball{r}{x}:= \oball{r}{X,x}&:=& \{y\in X: d(x,y)<r \},\\
	\cball{r}{x}:= \cball{r}{X,x}&:=& \{y\in X: d(x,y)\leq r \}.
\end{eqnarray*}
The set $\oball{r}{x}$ (resp. $\cball{r}{x}$) is called the \defstyle{open ball} (resp. \defstyle{closed ball}) of radius $r$ centered at $x$. Note that $\cball{r}{x}$ is closed, but is not necessarily the closure of $\oball{r}{x}$ in $X$. The metric space $X$ is \defstyle{boundedly compact} if every closed ball in $X$ is compact.

The rest of the paper is focused on \defstyle{pointed metric spaces}, abbreviated by \defstyle{\pms{} spaces} (Remark~\ref{rem:nonpointed} explains the non-pointed case). Such a space is a pair $(X,o)$, where $X$ is a metric space and $o$ is a distinguished point of $X$ called the \defstyle{\rooot} (or the \defstyle{origin}). %
A \defstyle{pointed measured metric space}, abbreviated by a \defstyle{\pmm{} space}, is a tuple $\mathcal X=(X,o,\mu)$ where $X$ is a metric space, %
$\mu$ is a non-negative Borel measure on $X$ and $o$ is a distinguished point of $X$. The balls centered at $o$ in $\mathcal X$ form other \pmm{} spaces as follows:

\begin{eqnarray*}
	\poball{\mathcal X}{r}&:=&\left(\oball{r}{o}, o, \restrict{\mu}{\oball{r}{o}} \right),\\
	\pcball{\mathcal X}{r}&:=&\left(\cball{r}{o}, o, \restrict{\mu}{\cball{r}{o}} \right).
\end{eqnarray*}

\begin{convention}
	All measures in this paper are Borel measures.
	A \pmm{} space $\mathcal X=(X,o,\mu)$ is called \defstyle{compact} if $X$ is compact and $\mu$ is a finite measure. Also, $\mathcal X$ is called \defstyle{boundedly compact} if $X$ is boundedly compact and $\mu$ is \defstyle{boundedly finite}; i.e., every ball in $X$ has finite measure under $\mu$.  
\end{convention}

%
%

%
A \defstyle{pointed isometry} $\rho:(X,o)\to (X',o')$ between two \pms{} spaces $(X,o)$ and $(X',o')$ is an isometry $\rho:X\rightarrow X'$ such that $\rho(o)=o'$.
A \defstyle{GHP-isometry} between two \pmm{} spaces $(X,o,\mu)$ and $(X',o',\mu')$ is a pointed isometry $\rho:(X,o)\rightarrow (X',o')$ such that $\rho_*\mu=\mu'$. 
If there exists a GHP-isometry between $(X,o,\mu)$ and $(X',o',\mu')$, then they are called \defstyle{GHP-isometric}. %

Let $\mstar$ be the set of equivalence classes of boundedly compact \pms{} spaces under pointed isometries\footnote{The $*$ sign stands for {`pointed'} and is included in the symbol mainly for compatibility with the literature.}. Define $\mcstar$ similarly by considering only compact spaces.
Also, let $\mmstar$ be the set of equivalence classes of boundedly compact \pmm{} spaces under GHP-isometries and define $\mmcstar$ similarly by considering only compact \pmm{} spaces. %
It can be seen that they are indeed sets.
%
%
%

\del{\begin{remark}
	\label{rem:GHspecialcase}
	Every pointed metric space $(X,o)$ can be regarded as a \pmm{} space by equipping it with the zero measure on $X$. So $\mstar$ can be identified with the subset $\{(X,o,\mu)\in\mmstar: \mu=0 \}$ of $\mmstar$. Similarly, $\mcstar$ can be identified with a subset of $\mmcstar$. 
	\\
	If the reader is interested only in pointed metric spaces rather than \pmm{} spaces, he or she can restrict the definitions to the subset $\{(X,o,\mu)\in\mmstar: \mu=0 \}$ of $\mmstar$ to simplify the arguments.
\end{remark}}

\begin{lemma}
	\label{lem:cadlag}
	Let $\mathcal X=(X,o,\mu)$ be a boundedly-compact \pmm{} space.
	\begin{enumerate}[(i)]
		\item The curve $t\mapsto \cball{t}{o}$ is c\`adl\`ag under the Hausdorff metric and its left limit at $t=r$ is the closure of $\oball{r}{o}$.
		\item The curve $t\mapsto \restrict{\mu}{\cball{t}{o}}$ is c\`adl\`ag under the Prokhorov metric and its left limit at $t=r$ is $\restrict{\mu}{\oball{r}{o}}$.
	\end{enumerate}
\end{lemma}

In fact, it will be seen that the curve $t\mapsto\pcball{\mathcal X}{t}$ is c\`adl\`ag under the Gromov-Hausdorff-Prokhorov metric (see Lemma~\ref{lem:cadlag2}).

\begin{proof}
	Let $r\geq 0$ and $\epsilon>0$. By compactness of the balls, it is straightforward to show that there exists $\delta>0$ such that 
	\[
	\cnei{\epsilon}{\cball{r}{o}}\supseteq \cball{r+\delta}{o},\quad 
	\mu(\cball{r}{o})+\epsilon\geq \mu(\cball{r+\delta}{o}).
	\]
	This implies that 
	\[
	\hausdorff(\cball{r}{o},\cball{r+\delta}{o})\leq \epsilon,\quad
	\prokhorov(\restrict{\mu}{\cball{r}{o}}, \restrict{\mu}{\cball{r+\delta}{o}})\leq\epsilon.
	\]
	It follows that the curves $t\mapsto \cball{t}{o}$ and $t\mapsto \restrict{\mu}{\cball{t}{o}}$ are right-continuous. Similarly, one can see that $\delta$ can be chosen such that 
	\[
	\cnei{\epsilon}{\cball{r-\delta}{o}}\supseteq \overline{\oball{r}{o}},\quad 
	\mu(\cball{r-\delta}{o})+\epsilon\geq \mu(\oball{r}{o}).
	\]
	Since $\cball{r-\delta}{o}\subseteq \overline{\oball{r}{o}}$, it follows that
	\[
	\hausdorff(\overline{\oball{r}{o}},\cball{r-\delta}{o})\leq \epsilon,\quad
	\prokhorov(\restrict{\mu}{\oball{r}{o}}, \restrict{\mu}{\cball{r-\delta}{o}})\leq\epsilon.
	\]
	This shows that the left limits of the curves are as desired and the claim is proved.
\end{proof}

\begin{definition}
	\label{def:continuityRadius}
	Let \unwritten{It can be seen that continuity radii under HP are exactly continuity radii under GHP.}
	$\mathcal X=(X,o,\mu)$ be a boundedly-compact \pmm{} space. A real number $r>0$ is called a \defstyle{continuity radius} for $\mathcal X$ if 
	$\cball{r}{o}$ is the closure of $\oball{r}{o}$ in $X$ and $\mu\left(\cball{r}{o}\setminus\oball{r}{o}\right)=0$. 
	Otherwise, it is called a \defstyle{discontinuity radius} for $\mathcal X$.
	Equivalently, $r$ is a continuity radius for $\mathcal X$ if and only if the curves $t\mapsto\cball{t}{o}$ and $t\mapsto \restrict{\mu}{\cball{t}{o}}$ (equivalently, the curve $t\mapsto\pcball{\mathcal X}{t}$) are continuous at $t=r$.
\end{definition}

\begin{lemma}
	\label{lem:continuityRadius}
	Every boundedly-compact \pmm{} space has at most countably many discontinuity radii.
\end{lemma}

\begin{proof}
	Every c\`adl\`ag function in a metric space has at most countably many discontinuity points. So the claim is implied by Lemma~\ref{lem:cadlag}.
\end{proof}

\subsection{The Metric in the Compact Case}
\label{subsec:GHPcompact}

In this subsection, the compact case of the Gromov-Hausdorff-Prokhorov metric is recalled from~\cite{AbDeHo13}. A Strassen-type result is also presented for the Gromov-Hausdorff-Prokhorov metric (Theorem~\ref{thm:GHPcompact}). In addition, the notion of \textit{\pmm{}-subspace} (Definition~\ref{def:measuredSubspace}) is introduced  and its properties are studied. The latter will be used in the next subsection.

Recall that  $\mmcstar$ is the set of (equivalence classes of) compact \pmm{} spaces.
For compact \pmm{} spaces $\mathcal X=(X,o_X,\mu_X)$ and $\mathcal Y=(Y,o_Y,\mu_Y)$, define the \defstyle{(compact) Gromov-Hausdorff-Prokhorov distance} of $\mathcal X$ and $\mathcal Y$, abbreviated here by the \defstyle{cGHP distance}, by
\begin{equation}
\label{eq:GHPcompact}
\cghp(\mathcal X,\mathcal Y):=\inf\{d(f(o_X),g(o_Y)) \vee \hausdorff(f(X), g(X)) \vee \prokhorov(f_*\mu_X, g_*\mu_Y ) \},
\end{equation}
where the infimum is over all metric spaces $Z$ and all isometric embeddings $f:X\rightarrow Z$ and $g:Y\rightarrow Z$.

The Gromov-Hausdorff-Prokhorov distance is define in~\cite{bookVi10} and~\cite{Mi09} for non-pointed metric spaces and in the case where $\mu_X$ and $\mu_Y$ are probability measures. The general case of the metric is defined in~\cite{AbDeHo13} by a similar formula in which $+$ is used instead of $\vee$, but is equivalent to~\eqref{eq:GHPcompact} up to a factor of 3. It is proved in~\cite{AbDeHo13} that $\cghp$ is a metric on $\mmcstar$ and makes it a complete separable metric space. The same proofs work by considering the slight modification mentioned above\del{(these results are also implied by those of Subsection~\ref{subsec:category-compact} below)}.
The reason to consider $\vee$ instead of $+$ is 
to ensure a Strassen-type result (Theorem~\ref{thm:GHPcompact} below)
that provides a useful formulation of the cGHP metric in terms of  approximate couplings and \textit{correspondences}.

\begin{remark}[Non-Pointed Spaces]
	\label{rem:nonpointed}
	In the compact case, a similar metric is defined between non-pointed spaces. It is obtained by removing the term $d(f(o_X),g(o_Y))$ from~\eqref{eq:GHPcompact}. Equivalently, by letting the distance of $(X,\mu_X)$ and $(Y,\mu_Y)$ be
	\[
	\min\left\{\cghp\Big((X,x,\mu_X), (Y,y,\mu_Y) \Big): x\in X, y\in Y \right\}.
	\]
	The results of this subsection have analogues for non-pointed spaces as well. However, considering pointed spaces is essential in the non-compact case discussed in the next subsection.
\end{remark}

\del{\begin{remark}[Gromov-Hausdorff Metric]
	\label{rem:GHspecialcase2}
	\mar{Duplicate with introduction}
	The \defstyle{Gromov-Hausdorff metric} $\cgh$ is defined on the space $\mcstar$ of compact pointed metric spaces similarly to~\eqref{eq:GHPcompact} by removing the term $\prokhorov(f_*\mu_X, g_*\mu_Y )$, or equivalently, by letting $\mu_X$ and $\mu_Y$ be the zero measures on $X$ and $Y$ respectively. The Gromov-Hausdorff metric is also defined on the space $\mcstarzero$ of non-pointed compact metric spaces similarly. More discussion is provided in Subsection~\ref{subsec:GH}.
\end{remark}}

A \defstyle{correspondence} $R$ (see e.g.,~\cite{bookBBI}) between $X$ and $Y$ is a relation between points of $X$ and $Y$ such that it is a Borel subset of $X\times Y$ and every point in $X$ corresponds to at least one point in $Y$ and vice versa. The \defstyle{distortion} of $R$ is
\[
\distortion(R):= \sup\{\norm{d(x,x')-d(y,y')}: (x,y)\in R, (x',y')\in R \}.
\]
The following is the main result of this subsection. It is a Strassen-type result for the metric $\cghp$ and is based on Theorem~\ref{thm:strassen}. %
\begin{theorem}
	\label{thm:GHPcompact}
	Let $\mathcal X=(X,o_X,\mu_X)$ and $\mathcal Y=(Y,o_Y,\mu_Y)$ be compact \pmm{} spaces and $\epsilon\geq 0$.
	\begin{enumerate}[(i)]
		\item \label{thm:GHPcompact:1} $\cghp(\mathcal X,\mathcal Y)\leq \epsilon$ if and only if there exists a correspondence $R$ between $X$ and $Y$ and a Borel measure $\alpha$ on $X\times Y$ such that $(o_X, o_Y)\in R$, $\distortion(R)\leq 2\epsilon$ and $D(\alpha;\mu_X,\mu_Y)+ \alpha(R^c) \leq \epsilon$.\\
		
		\item \label{thm:GHPcompact:2} In other words,
		\begin{equation}
		\label{eq:thm:GHPcompact}
		\cghp(\mathcal X,\mathcal Y)=\inf_{R,\alpha} \left\{\frac 1 2 \distortion(R) \vee \big(D(\alpha;\mu_X,\mu_Y)+ \alpha(R^c) \big)\right\}
		\end{equation}
		and the infimum is attained.
		
		\item \label{thm:GHPcompact:3} In addition, if $||\mu_X||\leq ||\mu_Y||$, then the infimum is attained for some $R$ and $\alpha$ such that $\pi_{1*}\alpha=\mu$ and $\pi_{2*}\alpha\leq \nu$.
	\end{enumerate}
\end{theorem}

\begin{remark}
	The formula~\eqref{eq:thm:GHPcompact} resembles the definition of a metric in~\cite{AdBrGoMi17} which uses $\vee$ instead of $+$. The definition in~\cite{AdBrGoMi17}, although is not equal to the classical Gromov-Hausdorff-Prokhorov metric, but is equivalent to it.
\end{remark}

\begin{remark}
	Theorem~\ref{thm:GHPcompact} generalizes Theorem~7.3.25 of~\cite{bookBBI} and Proposition~6 of~\cite{Mi09}. The former is a result for the Gromov-Hausdorff distance; i.e., the case where  $\mu_X$ and $\mu_Y$ are the zero measures. The latter is the case where $\mu_X$ and $\mu_Y$ are probability measures, where $\alpha$ can be chosen to be a coupling of $\mu_X$ and $\mu_Y$ and the term $D(\alpha;\mu_X,\mu_Y)$ disappears.
\end{remark}

\begin{proof}[Proof of Theorem~\ref{thm:GHPcompact}]
	Assume $R$ is a correspondence such that $(o_X,o_Y)\in R$ and $\distortion(R)\leq 2\epsilon$.  %
	By Theorem~7.3.25 in~\cite{bookBBI}, without loss of generality, one can assume $X,Y\subseteq Z$, $\hausdorff(X,Y)\leq \epsilon$ and if $(x,y)\in R$, then $d(x,y)\leq \epsilon$. Assume $\alpha$ is a measure such that $D(\alpha;\mu_X,\mu_Y)+ \alpha(R^c)\leq \epsilon$. One has $\alpha(\{(x,y):d(x,y)>\epsilon\}) \leq \alpha(R^c)$. So, Theorem~\ref{thm:strassen} implies that $\prokhorov(\mu_X,\mu_Y)\leq \epsilon$. This implies that $\cghp(\mathcal X,\mathcal Y)\leq \epsilon$.

	Conversely, assume $\cghp(\mathcal X,\mathcal Y)\leq \epsilon$. Let $\delta>\epsilon$. By~\eqref{eq:GHPcompact}, one can find two isometric embeddings $f:X\rightarrow Z$ and $g:Y\rightarrow Z$ for some $Z$ such that %
	\begin{equation}
	\label{eq:prop:GHPcompact:1}
	\left\{
	\begin{array}{lll}
	d(f(o_X),g(o_Y))&\leq&\delta,\\
	\hausdorff(f(X),g(Y))&\leq&\delta, \\
	\prokhorov(f_*\mu_X,g_*\mu_Y)&\leq& \delta,
	\end{array}
	\right.
	\end{equation}
	where $d_H$ and $\prokhorov$ are defined using this metric on $Z$.
	Let $R_{\delta}:=\{(x,y)\in X\times Y: d(f(x),g(y))\leq \delta \}$. The first condition in~\eqref{eq:prop:GHPcompact:1} implies that $(o_X,o_Y)\in R_{\delta}$. The second condition in~\eqref{eq:prop:GHPcompact:1} implies that $R_{\delta}$ is a correspondence. One also has $\distortion(R_{\delta})\leq 2\delta$. The third condition in~\eqref{eq:prop:GHPcompact:1} and Theorem~\ref{thm:strassen} imply that there exists a measure $\beta$ on $Z\times Z$ such that $D(\beta;f_*\mu_X,g_*\mu_Y)+ \beta(\{(x,y)\in Z\times Z:d(x,y)>\delta\})\leq \delta$. The third part of Theorem~\ref{thm:strassen} shows that $\beta$ can be chosen to be supported on $f(X)\times g(Y)$. Therefore, $\beta$ induces a measure $\alpha_{\delta}$ on $X\times Y$ by the inverses of the isometries $f$ and $g$. Thus, 
	\begin{equation}
	\label{eq:prop:GHPcompact:2}
	D(\alpha_{\delta};\mu_X,\mu_Y)+ \alpha_{\delta}(R_{\delta}^c)\leq \delta.
	\end{equation}
	
	Now, we will consider the limits of $R_{\delta}$ and $\alpha_{\delta}$ as $\delta\downarrow\epsilon$. Since $X\times Y$ is compact, Blaschke's theorem (see e.g., Theorem~7.3.8 in~\cite{bookBBI}) implies that 
	there exists a subsequence of the sets $R_{\delta}$ that is convergent in the Hausdorff metric to some closed subset of $X\times Y$. Let $R\subseteq X\times Y$ be the limit of this sequence.
	Since each $R_{\delta}$ is a correspondence, it can be seen that $R$ is also a correspondence and $(o_X,o_Y)\in R$. Also, it can be seen that the fact $\distortion(R_{\delta})\leq 2\delta$ implies that $\distortion(R)\leq 2\epsilon$. Prokhorov's theorem on tightness~\cite{Pr56} (see also~\cite{bookBi99} or~\cite{bookKa17randommeasures}) implies that there is a further subsequence such that the measures $\alpha_{\delta}$ converge weakly. So assume $\alpha_{\delta}\to \alpha$ along this subsequence. From now on, we assume $\delta$ is always in the subsequence without mentioning it explicitly. %
	
	Let $h$ be any continuous function on $X\times Y$ whose support is disjoint from $R$ and $h\leq 1$. This implies that $\mathrm{supp}(h)\cap R_{\delta}=\emptyset$ for sufficiently small $\delta$. Therefore, $\int hd\alpha_{\delta} \leq \alpha_{\delta}(R_{\delta}^c)$. The weak convergence $\alpha_{\delta}\to\alpha$ gives $\int hd\alpha\leq \liminf \alpha_{\delta}(R_{\delta}^c)$. %
	By considering this for all $h$, one gets 
	\begin{equation}
	\label{eq:prop:GHPcompact:3}
	\alpha(R^c)\leq \liminf \alpha_{\delta}(R_{\delta}^c).
	\end{equation}
	
	For considering the discrepancy $D(\alpha;\mu_X,\mu_Y)$ of $\alpha$, assume $\beta$ is chosen in the above argument such that the condition in part~\eqref{thm:Prokhorov:2} of Theorem~\ref{thm:strassen} is satisfied, hence $\pi_{1*}\alpha_{\delta}= \mu_X$ and $\pi_{2*}\alpha_{\delta}\leq \mu_Y$. One can easily obtain  $\pi_{1*}\alpha= \mu_X$ and $\pi_{2*}\alpha\leq \mu_Y$. Therefore, one gets
	\begin{eqnarray*}
		D(\alpha_{\delta}; \mu_X,\mu_Y) &=& \mu_Y(Y)-\alpha_{\delta}(X\times Y),\\ %
		D(\alpha; \mu_X,\mu_Y) &=& \mu_Y(Y)-\alpha(X\times Y).%
	\end{eqnarray*}
	These equations enable us to obtain that $D(\alpha; \mu_X,\mu_Y) = \lim 	D(\alpha_{\delta}; \mu_X,\mu_Y)$. Finally,  \eqref{eq:prop:GHPcompact:2} and~\eqref{eq:prop:GHPcompact:3} imply that $D(\alpha;\mu_X,\mu_Y) + \alpha(R^c)\leq \epsilon$. Therefore, $R$ and $\alpha$ satisfy the claim. This proves parts~\eqref{thm:GHPcompact:1} and~\eqref{thm:GHPcompact:2} of the theorem.
	
	As mentioned above, if $\beta$ is chosen such that $\pi_{1*}\beta=f_*\mu_X$ and $\pi_{2*}\beta\leq g_*\mu_Y$, then the claim of part~\eqref{thm:GHPcompact:3} is obtained. So the proof is completed.
\end{proof}

Theorem~\ref{thm:GHPcompact} readily implies the following.

\begin{corollary}
	\label{cor:cghp-inf}
	The infimum in the definition of the cGHP metric~\eqref{eq:GHPcompact} is attained.
\end{corollary}

The following are further properties of $\cghp$ which are needed later.

\begin{lemma}
	\label{lem:radius}
	For compact \pmm{} spaces $\mathcal X=(X,o_X,\mu_X)$ and $\mathcal Y=(Y,o_Y,\mu_Y)$, 
	\[
	\max\{d(o_Y,y):y\in Y \}\leq \max\{d(o_X,x):x\in X\} + 2\cghp(\mathcal X, \mathcal Y).
	\]
\end{lemma}
\begin{proof}
	Let $\epsilon:=\cghp(\mathcal X, \mathcal Y)$. By Theorem~\ref{thm:GHPcompact}, there is a correspondence $R$ between $X$ and $Y$ such that $(o_X,o_Y)\in R$ and $\distortion(R)\leq 2\epsilon$. Let $y\in Y$ be arbitrary. There exists $x\in X$ that $R$-corresponds to $y$. Since $\distortion(R)\leq 2\epsilon$, one gets $d(o_Y,y)\leq d(o_X,x)+2\epsilon$. This implies the claim.
\end{proof}

The following definition and results are needed for the next subsection.

\begin{definition}
	\label{def:measuredSubspace}
	Let $\mathcal X=(X,o,\mu)$ and $\mathcal X'=(X',o',\mu')$ be \pmm{} spaces. $\mathcal X'$ is called a \defstyle{\pmm-subspace} of $\mathcal X$ if $X'\subseteq X$, $o'=o$ and $\mu'\leq \mu$.
	The following symbol is used to express that $\mathcal X'$ is a \pmm-subspace of $\mathcal X$:
	\[
	\mathcal X'\preceq \mathcal X.
	\]
	For two \pmm{}-subspaces $\mathcal X_i=(X_i,o,\mu_i)$ of $\mathcal X$ ($i=1,2$), their \defstyle{Hausdorff-Prokhorov distance} is defined by
	\begin{equation}
	\label{eq:HP}
	\hp(\mathcal X_1,\mathcal X_2):=d_H(X_1,X_2)\vee \prokhorov(\mu_1,\mu_2).
	\end{equation}
	This equation immediately gives
	\begin{equation}
	\label{eq:hp-ineq}
	\cghp(\mathcal X_1,\mathcal X_2)\leq \hp(\mathcal X_1,\mathcal X_2).
	\end{equation}
\end{definition}

\begin{lemma}
	\label{lem:GHP-monotone}
	Let $\mathcal X$ and $\mathcal Y$ be compact \pmm{} spaces.
	\begin{enumerate}[(i)]
		\item \label{part:lem:GHP-monotone:1} If $\mathcal X'$ is a compact \pmm-subspace of $\mathcal X$, then there exists a compact \pmm-subspace $\mathcal Y'$ of $\mathcal Y$ such that 
		\[
		\cghp(\mathcal X', \mathcal Y')\leq \cghp(\mathcal X, \mathcal Y). %
		\]
		\item \label{part:lem:GHP-monotone:2} 
		Let $\epsilon:= \cghp(\mathcal X,\mathcal Y)$ and $r\geq 2\epsilon$ be arbitrary. 
		If in addition to~\eqref{part:lem:GHP-monotone:1}, one has $\pcball{\mathcal X}{r}\preceq \mathcal X'\preceq \mathcal X$, then $\mathcal Y'$ can be chosen such that $\pcball{\mathcal Y}{r-2\epsilon}\preceq \mathcal Y'\preceq\mathcal Y$.
		%
		%
	\end{enumerate}

\end{lemma}

\begin{proof}
	Let $\mathcal X=:(X,o_X,\mu_X)$, $\mathcal Y=:(Y,o_Y,\mu_Y)$ and $\epsilon:=\cghp(\mathcal X, \mathcal Y)$.
	By Theorem~\ref{thm:GHPcompact}, there exists a correspondence $R$ between $X$ and $Y$ and a measure $\alpha$ on $X\times Y$ such that $(o_X, o_Y)\in R$, $\distortion(R)\leq 2\epsilon$ and $D(\alpha;\mu_X,\mu_Y)+ \alpha(R^c) \leq \epsilon$. By part~\eqref{thm:GHPcompact:3} of the theorem, we may assume $\pi_{1*}\alpha\leq \mu_X$ and $\pi_{2*}\alpha\leq \mu_Y$. 
	Also, by replacing $R$ with its closure in $X\times Y$ if necessary, we might assume $R$ is closed without loss of generality. Let $X'=:(X',o_X,\mu'_X)$. 
	
	\begin{figure}
		\begin{center}
			\includegraphics[width=.5\textwidth]{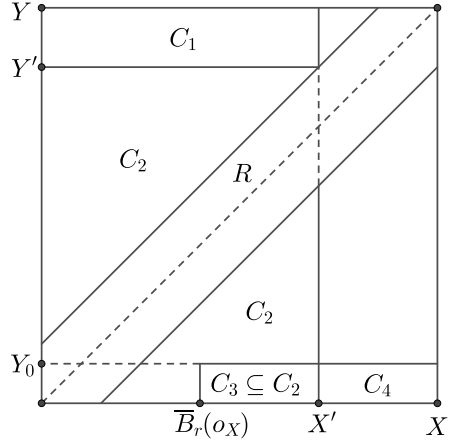}
			\caption{
				A schematic picture of the sets in the proof of Lemma~\ref{lem:GHP-monotone}. 
				The sets $X,Y,X',Y',Y_0$ and $\cball{r}{o_X}$ are depicted as intervals one of whose end points is the lower left corner of the figure and the other end is shown by a label.
			}
			\label{fig:ghp-monotone}
		\end{center}
	\end{figure}
	
	\eqref{part:lem:GHP-monotone:1}. Let $Y'$ be the set of points in $Y$ that $R$-correspond to some point in $X'$. Let $\alpha_1:=\restrict{\alpha}{X'\times Y'}$. By Lemma~1 of~\cite{Th96}, there exists a measures $\alpha'\leq \alpha_1$ on $X'\times Y'$ such that $\pi_{1*}\alpha'=\pi_{1*}\alpha_1\wedge \mu'_X$. Consider the measure $\mu'_Y:=\pi_{2*}\alpha'$ on $Y'$. We claim that $\mathcal Y':=(Y',o_Y,\mu'_Y)$ satisfies the desired property.
	Note that $Y'$ is a closed subset of $Y$, $o_Y\in Y'$ and $\mu'_Y\leq \pi_{2*}\alpha_1\leq \pi_{2*}\alpha\leq \mu_Y$. So $\mathcal Y'\preceq \mathcal Y$. Let $R':=R\cap (X'\times Y')$. The definition of $Y'$ gives that $R'$ is a correspondence between $X'$ and $Y'$ and $(o_X,o_Y)\in R'$. Also, it is clear that $\distortion(R')\leq \distortion(R)\leq 2\epsilon$. By Theorem~\ref{thm:GHPcompact}, it remains to prove that 
	\begin{equation}
	\label{eq:lem:GHP-monotone:0} 
	D(\alpha';\mu'_X,\mu'_Y)+ \alpha'((R')^c) \leq \epsilon.
	\end{equation}
	
	Let $C_1:=X'\times (Y\setminus Y')$ %
	and $C_2:=(X'\times Y') \setminus R'$ (see Figure~\ref{fig:ghp-monotone}). One has $\alpha'((R')^c) \leq \alpha(C_2)$. Since $\pi_{1*}\alpha'\leq \mu'_X$, one gets 
	\[
	\tv{\pi_{1*}\alpha'}{\mu'_X} = ||\mu'_X||-||\pi_{1*}\alpha'|| = ||\mu'_X||-||\pi_{1*}\alpha_1\wedge \mu'_X||.
	\]
	Since $\mu'_X$ and $\pi_{1*}\alpha_1$ are bounded by $\restrict{\mu_X}{X'}$, one can easily deduce that
	\begin{eqnarray}
	\nonumber \tv{\pi_{1*}\alpha'}{\mu'_X}&\leq& ||\restrict{\mu_X}{X'}||-||\pi_{1*}\alpha_1||\\
	\nonumber &=&\mu_X(X')-\alpha_1(X'\times Y')\\
	\nonumber &=& \mu_X(X')-\alpha(X'\times Y)+\alpha(C_1)\\
	\label{eq:lem:GHP-monotone:1} &\leq& \tv{\pi_{1*}\alpha}{\mu_X} + \alpha(C_1).
	\end{eqnarray}
	Since $\tv{\pi_{2*}\alpha'}{\mu'_Y}=0$, one gets that 
	\[
	D(\alpha';\mu'_X,\mu'_Y)\leq \tv{\pi_{1*}\alpha}{\mu_X} + \alpha(C_1)\leq D(\alpha;\mu_X,\mu_Y)+\alpha(C_1).
	\]
	Therefore,
	\begin{eqnarray*}
		D(\alpha';\restrict{\mu_X}{X'},\mu'_Y) + \alpha'((R')^c) &\leq & D(\alpha; \mu_X, \mu_Y) + \alpha(C_1\cup C_2)\\
		&\leq & D(\alpha; \mu_X, \mu_Y) + \alpha(R^c)\\
		&\leq & \epsilon,
	\end{eqnarray*}
	where the first inequality is because $C_1\cap C_2=\emptyset$ and the second inequality is because $C_1$ and $C_2$ are disjoint from $R$, which is easy to see. So, \eqref{eq:lem:GHP-monotone:0} is proved and the proof is completed.
	
	\eqref{part:lem:GHP-monotone:2}. 
	Let $Y_0:=\cball{r-2\epsilon}{o_Y}$. Define $Y', R', \alpha_1$ and $\alpha'$ as in part~\eqref{part:lem:GHP-monotone:1} and replace $\mu'_Y$ by $\mu''_Y:=\pi_{2*}\alpha'\vee \restrict{\mu_Y}{Y_0}$. 
	Let $y\in Y_0$ be arbitrary. Since $R$ is a correspondence, there exists $x\in X$ such that $(x,y)\in R$. Since $\distortion(R)\leq 2\epsilon$, one gets that $d(x,o_X)\leq d(y,o_Y)+2\epsilon\leq r$. This implies that $x\in \cball{r}{o_X}\subseteq X'$. The definition of $Y'$ implies that $y\in Y'$. Hence, $Y'\supseteq Y_0$ and so $\mu''_Y$ is supported on $Y'$. We will show that $\mathcal Y'':=(Y',o_Y,\mu''_Y)$ satisfies the claim.
	Note that  $(\restrict{\mu_Y}{Y_0}) \leq \mu''_Y \leq (\restrict{\mu_Y}{Y'})$. This gives that $\pcball{\mathcal Y}{r-2\epsilon}\preceq\mathcal Y''\preceq\mathcal Y$.
	
	Define $C_1$ and $C_2$ as in part~\eqref{part:lem:GHP-monotone:1}. 
	The proof of part~\eqref{part:lem:GHP-monotone:1} shows that $(o_X,o_Y)\in R'$, $\distortion(R')\leq \distortion(R)\leq 2\epsilon$, $\alpha'((R')^c)=\alpha'(C_2)$ and~\eqref{eq:lem:GHP-monotone:1} holds. To bound $\tv{\pi_{2*}\alpha'}{\mu''_Y}$, note that $\pi_{2*}\alpha'\leq \restrict{\mu_Y}{Y_0}$ on $Y_0$. So the definition of $\mu''_Y$ gives that
	\begin{eqnarray*}
		\tv{\pi_{2*}\alpha'}{\mu''_Y} &=& \mu_Y(Y_0)-\pi_{2*}\alpha'(Y_0)\\
		&=& \mu_Y(Y_0)-\alpha'(X'\times Y_0)\\
		&= & \mu_Y(Y_0)-\alpha'(\cball{r}{o_X}\times Y_0) - \alpha'(C_3),
	\end{eqnarray*}
	where $C_3:=(X'\setminus\cball{r}{o_X})\times Y_0$.
	Since $\mu'_X$ agrees with $\mu_X$ on $\cball{r}{o_X}$, one gets that $\pi_{1*}\alpha_1\leq \mu'_X$ on $\cball{r}{o_X}$. So the definition of $\alpha'$ implies that $\pi_{1*}\alpha'=\pi_{1*}\alpha_1$ on $\cball{r}{o_X}$. The condition $\alpha'\leq \alpha_1$ gives that $\alpha'=\alpha_1=\alpha$ on $ \cball{r}{o_X}\times Y'$. So, by letting $C_4:=(X\setminus X')\times Y_0$, the above equation gives
	\begin{eqnarray*}
		\tv{\pi_{2*}\alpha'}{\mu''_Y}&=& \mu_Y(Y_0)-\alpha(\cball{r}{o_X}\times Y_0) - \alpha'(C_3)\\
		&=& \mu_Y(Y_0)-\alpha(X\times Y_0) + \alpha(C_3\cup C_4) - \alpha'(C_3)\\
		&\leq& \tv{\mu_Y}{\pi_{2*}\alpha} + \alpha(C_3\cup C_4) - \alpha'(C_3).
	\end{eqnarray*}
	The above discussions show that $C_3\cap R=\emptyset$, which implies that $C_3\subseteq C_2$. Also, note that the four sets $C_1,C_2,C_4,R$ are pairwise disjoint. So, by summing up, we get
	\begin{eqnarray*}
		D(\alpha';\mu'_X,\mu''_Y) + \alpha'((R')^c) &\leq & \tv{\mu_X}{\pi_{1*}\alpha} + \alpha (C_1) + \\ 
		&&  \tv{\mu_Y}{\pi_{2*}\alpha}  + \alpha(C_3\cup C_4) - \alpha'(C_3) + \alpha'(C_2)\\
		&\leq &
		D(\alpha; \mu_X, \mu_Y) + \alpha(C_1\cup C_3 \cup C_4) + \alpha'(C_2\setminus C_3)\\
		&\leq &
		D(\alpha; \mu_X, \mu_Y) + \alpha(C_1\cup C_2 \cup C_4)\\
		&\leq & D(\alpha; \mu_X, \mu_Y) + \alpha(R^c)\\
		&\leq & \epsilon.
	\end{eqnarray*}
	Finally, Theorem~\ref{thm:GHPcompact} implies that  
	$\cghp(\mathcal X',\mathcal Y'')\leq \epsilon$ and the claim is proved.
\end{proof}

\begin{lemma}
	\label{lem:GHP-subsets}
	If $\mathcal X$ is a compact \pmm{} space, then the set of compact \pmm-subspaces of $\mathcal X$ is compact under the topology of the metric $\cghp$. 
\end{lemma}

\begin{proof}
	By~\eqref{eq:hp-ineq}, it is enough to show that the set of compact \pmm-subspaces of $\mathcal X$ is compact under the metric $\hp$. Let $\mathcal X=:(X,o,\mu)$ and consider a sequence $\mathcal X_n=(X_n,o,\mu_n)$ of \pmm{}-subspaces of $\mathcal X$.
	Blaschke's theorem (see e.g., Theorem~7.3.8 in~\cite{bookBBI}) implies that the set of compact subsets of $X$ is compact under $\hausdorff$. Also, the set of measures on $X$ which are bounded by $\mu$ is tight and closed (under weak convergence). So Prokhorov's theorem implies that the latter is compact. %
	So by passing to a subsequence, one may assume that $\hausdorff(X_n,Y)\to 0$ and $\prokhorov(\mu_n,\nu)\to 0$ for some compact subset $Y\subseteq X$ and some measure $\nu\leq \mu$. It is left to the reader to show that $o\in Y$ and $\nu$ is supported on $Y$. This implies that $\hp(\mathcal X_n, (Y,o,\nu))\to 0$ and the claim is proved.
\end{proof}

\subsection{The Metric in the Boundedly-Compact Case}
\label{subsec:ghp-nonCompact}
This subsection presents the definition of the Gromov-Hausdorff-Prokhorov metric in the boundedly-compact case and proves that it is indeed a metric. Meanwhile, K\"onig's infinity lemma is generalized to compact sets (Lemma~\ref{lem:infinity}) and is used in the proofs. The Gromov-Hausdorff metric is a special case and will be discussed in Subsection~\ref{subsec:GH}. %

Let $\mathcal X$ and $\mathcal Y$
be {boundedly-compact} \pmm{} spaces. According to the heuristic mentioned in the introduction, the idea is that $\mathcal X$ and $\mathcal Y$ are \textit{close} if two \textit{large} compact portions of the two spaces are close under the metric $\cghp$. 
In the definition, for a fixed large $r$, the ball $\pcball{\mathcal X}{r}$ is not needed to be close to $\pcball{\mathcal Y}{r}$ due to the points that are close to the boundaries of the balls. Instead, the former should be close to a \textit{perturbation} of the latter.
This is made precise in the following (see Remark~\ref{rem:variant} for another definition and also Theorem~\ref{thm:convergence}). 

For $r\geq \epsilon\geq 0$, define
\begin{equation}
\label{eq:a_r}
a(\epsilon,r;\mathcal X,\mathcal Y):=\inf\{\cghp(\pcball{\mathcal X}{r}, \mathcal Y')\},
\end{equation}
where the infimum is over all compact \pmm-subspaces $\mathcal Y'$ of $\mathcal Y$ (Definition~\ref{def:measuredSubspace}) such that $\pcball{\mathcal Y}{r-\epsilon}\preceq\mathcal Y'\preceq \mathcal Y$ (by removing the condition $\pcball{\mathcal Y}{r-\epsilon}\preceq\mathcal Y'$, all of the results will remain valid except maybe those in Subsection~\ref{subsec:convergence}). %
Lemma~\ref{lem:a-inf} below proves that the infimum is attained.
The case $r=1/\epsilon$ is mostly used in the following. So, for $0<\epsilon\leq 1$, define 
\[
a_{\epsilon}(\mathcal X, \mathcal Y):=a(\epsilon,1/\epsilon;\mathcal X, \mathcal Y).
\]
Of course, this is not a symmetric function of $\mathcal X$ and $\mathcal Y$. 
\begin{definition}
	\label{def:GHP}
	Let $\mathcal X$ and $\mathcal Y$ be boundedly-compact \pmm{} spaces. The \defstyle{Gromov-Hausdorff-Prokhorov (GHP) distance} of $\mathcal X$ and $\mathcal Y$ is defined by
	\begin{equation}
	\label{eq:GHP}
	\ghp(\mathcal X,\mathcal Y):= \inf\{ \epsilon\in (0,1]: 
	a_{\epsilon}(\mathcal X,\mathcal Y) \vee a_{\epsilon}(\mathcal Y,\mathcal X)
	< \frac{\epsilon}{2}\},
	\end{equation}
	with the convention that $\inf \emptyset:=1$.
\end{definition}
In fact, Lemma~\ref{lem:GHP-ineqs} below implies that the infimum in~\eqref{eq:GHP} is not attained. 
Note that we always have 
\begin{equation}
\label{eq:GHP<1}
0\leq \ghp(\mathcal X,\mathcal Y)\leq 1.
\end{equation}

The following theorem is the main result of this subsection. Further properties of the function $\ghp$ are discussed in the next subsections.

\begin{theorem}
	\label{thm:GHP-Metric}
	The GHP distance~\eqref{eq:GHP} induces a metric on $\mmstar$.
\end{theorem}

To prove this theorem, the following lemmas are needed.

\begin{lemma}[K\"onig's Infinity Lemma For Compact Sets]
	\label{lem:infinity}
	Let $C_n$ be a compact set for each $n\in\mathbb N$ and $f_n:C_n\rightarrow C_{n-1}$ be a continuous function for $n>1$. Then, there exists a sequence $x_1\in C_1, x_2\in C_2,\ldots$ such that $f_n(x_n)=x_{n-1}$ for each $n>1$.
\end{lemma}

This lemma is a generalization of K\"{o}nig's infinity lemma, which is the special case where each $C_n$ is a finite set.
\begin{proof}%
	Let $C_0$ be a single point and $f_1:C_1\rightarrow C_0$ be the unique function. For $m>n$, let $f_{m,n}:=f_{n+1}\circ \cdots \circ f_m$. Note that for every $n$, the sets $f_{m,n}(C_{m})$ for $m=n+1,n+2,\ldots$ are nested. We will define the sequence $x_n\in C_n$ inductively such that $x_n$ is in the image of $f_{m,n}$ for every $m>n$. Let $x_0:=0$ which has that property. Assuming $x_{n-1}$ is defined, let $x_{n}$ be an arbitrary point in the intersection of $f_n^{-1}(x_{n-1})$ and $\bigcap_{m=n+1}^{\infty}f_{m,n}(C_m)$ (note that the intersection is nonempty by compactness and the induction hypothesis). It can be seen that $x_n$ satisfies the induction claim and the lemma is proved.
\end{proof}

\begin{lemma}
	\label{lem:a-inf}
	The infimum in~\eqref{eq:a_r} is attained. 
\end{lemma}
\begin{proof}%
	The claim is implied by Lemma~\ref{lem:GHP-subsets} and the fact that $\cghp$ is a metric on $\mmcstar$.
\end{proof}

\begin{lemma}
	\label{lem:a}
	The number $a(\epsilon,r;\mathcal X,\mathcal Y)$ is non-increasing w.r.t. $\epsilon$. Moreover, if $a(\epsilon, r_0;\mathcal X,\mathcal Y)\leq\frac{\epsilon}{2}$, then $a(\epsilon,r;\mathcal X,\mathcal Y)$ is non-decreasing w.r.t. $r$ in the interval $r\in [\epsilon,r_0]$.
\end{lemma}
\begin{proof}%
	The first claim is easy to check. For the second claim, it is enough to prove that for $r\in [\epsilon, r_0)$, one has $a(\epsilon,r;\mathcal X,\mathcal Y)\leq a(\epsilon,r_0;\mathcal X,\mathcal Y)$. 
	
	Let $a(\epsilon,r_0;\mathcal X,\mathcal Y)=:\delta\leq\frac{\epsilon}{2}$. 
	By Lemma~\ref{lem:a-inf}, there is a compact \pmm-subspace $\mathcal Y'$ of $\mathcal Y$ such that $\pcball{\mathcal Y}{r_0-\epsilon}\preceq \mathcal Y'$ and $\cghp(\pcball{\mathcal X}{r_0}, \mathcal Y')\leq \delta$. 
	By Lemma~\ref{lem:GHP-monotone}, there is a further compact \pmm-subspace $\mathcal Y''$ of $\mathcal Y'$ 
	such that $\pcball{\mathcal Y'}{r-2\delta} \preceq \mathcal Y''$ and  $\cghp(\pcball{\mathcal X}{r}, \mathcal Y'')\leq \delta$. 
	Since $2\delta\leq \epsilon$ and $r<r_0$ by assumption, one gets that $\pcball{\mathcal Y}{r-\epsilon}\preceq\mathcal Y''$. 
	Therefore, $a(\epsilon,r;\mathcal X,\mathcal Y)\leq \delta$ by definition. This proves the claim.
\end{proof}

\begin{lemma} 
	\label{lem:GHP-ineqs}
	For $\delta:=\ghp(\mathcal X,\mathcal Y)\leq 1$, one has
	\[
	a_{\delta}(\mathcal X,\mathcal Y) \vee a_{\delta}(\mathcal Y,\mathcal X)\geq \frac{\delta}{2}.
	\]
	In addition, if $\ghp(\mathcal X,\mathcal Y)<\gamma\leq 1$, then
	\[
	a_{\gamma}(\mathcal X,\mathcal Y) \vee a_{\gamma}(\mathcal Y,\mathcal X)< \frac{\gamma}{2}.
	\]

\end{lemma}
\begin{proof}%
	For the first claim, assume that for $\delta:=\ghp(\mathcal X, \mathcal Y)$, one has $a_{\delta}(\mathcal X,\mathcal Y) \vee a_{\delta}(\mathcal Y,\mathcal X)<\frac{\delta}{2}-\alpha$, where $\alpha>0$. So there exists a compact \pmm-subspace $\pcball{\mathcal Y}{1/\delta-\delta}\preceq\mathcal Y'\preceq\mathcal Y$ such that $\cghp(\pcball{\mathcal X}{1/\delta},\mathcal Y')<\frac{\delta}{2}-\alpha$. 
	Let $\mathcal Y=:(Y,o_Y,\mu_Y)$ and $\mathcal Y'=:(Y',o_Y,\mu'_Y)$.
	Lemma~\ref{lem:cadlag} implies that there exists $\epsilon<\delta$ such that $\hp(\pcball{\mathcal X}{1/\epsilon},\pcball{\mathcal X}{1/\delta})<\frac{\alpha}{4}$. Let $Y_0:=\cball{1/\epsilon-\epsilon}{o_Y}$. By a similar argument to Lemma~\ref{lem:cadlag}, $\epsilon<\delta$ can be chosen such that $\hp(\mathcal Y',\mathcal Y'')<\frac{\alpha}{4}$, where $\mathcal Y'':=(Y'\cup Y_0,o_Y,\mu'_Y\vee \restrict{\mu_Y}{Y_0})$. %
	Note that $\pcball{\mathcal Y}{1/\epsilon-\epsilon}\preceq \mathcal Y''\preceq \mathcal Y$.
	The triangle inequality implies that $\hp(\pcball{\mathcal X}{1/\epsilon},\mathcal Y'')<\frac{\delta}{2}-\frac{\alpha}{2}$. 
	Now, if $\epsilon$ is chosen such that $\epsilon>\delta-\alpha$, the definition~\eqref{eq:a_r} gives that $a_{\epsilon}(\mathcal X, \mathcal Y)<\frac{\epsilon}{2}$. Similarly, $\epsilon<\delta$ can be chosen such that $a_{\epsilon}(\mathcal Y, \mathcal X)<\frac{\epsilon}{2}$. This gives that $\ghp(\mathcal X,\mathcal Y)\leq \epsilon<\delta$, which is a contradiction.
	
	For the second claim, since $\ghp(\mathcal X, \mathcal Y)<\gamma\leq 1$, \eqref{eq:GHP} implies that there exists $\epsilon<\gamma$ such that $a_{\epsilon}(\mathcal X,\mathcal Y) \vee a_{\epsilon}(\mathcal Y,\mathcal X)< \frac{\epsilon}{2}$. The second claim in Lemma~\ref{lem:a} implies that $a(\epsilon,1/\gamma;\mathcal X,\mathcal Y) \vee a(\epsilon,1/\gamma;\mathcal Y,\mathcal X)< \frac{\epsilon}{2}$. Therefore, the first claim in Lemma~\ref{lem:a} implies that $a_{\gamma}(\mathcal X,\mathcal Y) \vee a_{\gamma}(\mathcal Y,\mathcal X)< \frac{\epsilon}{2}<\frac{\gamma}{2}$. %

\end{proof}

\begin{proof}[Proof of Theorem~\ref{thm:GHP-Metric}]
	It is easy to see that $\ghp(\mathcal X,\mathcal Y)$ depends only on the isometry classes of $\mathcal X$ and $\mathcal Y$. Therefore, it induces a function on $\mmstar\times\mmstar$, which is denoted by the same symbol $\ghp$. It is immediate that $\ghp$ is symmetric and $\ghp(\mathcal X,\mathcal X)=0$.
	
	Let $\mathcal X,\mathcal Y$ and $\mathcal Z$ be boundedly compact \pmm{} spaces. Assume $\ghp(\mathcal X,\mathcal Y)<\epsilon$ and $\ghp(\mathcal Y,\mathcal Z)<\delta$. For the triangle inequality, it is enough to show that $\ghp(\mathcal X,\mathcal Z)\leq \epsilon+\delta$. If $\epsilon+\delta\geq 1$, the claim is clear by~\eqref{eq:GHP<1}. So assume $\epsilon+\delta<1$. By Lemma~\ref{lem:GHP-ineqs}, one gets that $a_{\epsilon}(\mathcal X,\mathcal Y)\vee a_{\epsilon}(\mathcal Y,\mathcal X)< \frac{\epsilon}{2}$ and $a_{\delta}(\mathcal Y,\mathcal Z)\vee a_{\delta}(\mathcal Z,\mathcal Y)< \frac{\delta}{2}$. Lemma~\ref{lem:a} implies that $a(\epsilon,1/(\epsilon+\delta);\mathcal X,\mathcal Y)<\frac{\epsilon}{2}$. 
	Therefore, by~\eqref{eq:a_r}, there is a compact \pmm{}-subspace $\mathcal Y'=(Y',o_Y,\mu'_Y)$ of $\mathcal Y$ such that $\pcball{\mathcal Y}{1/(\epsilon+\delta)-\epsilon}\preceq \mathcal Y'$ and 
	\[
	\cghp(\pcball{\mathcal X}{1/(\epsilon+\delta)}, \mathcal Y')<\frac{\epsilon}{2}.
	\]
	So, Lemma~\ref{lem:radius} implies that $\mathcal Y'\preceq\pcball{\mathcal Y}{1/{(\epsilon+\delta)}+\epsilon}$.
	It is straightforward to deduce from $0<\epsilon+\delta<1$ that $\epsilon+\frac 1{\epsilon+\delta}<\frac 1{\delta}$. Therefore, $\mathcal Y'\preceq\pcball{\mathcal Y}{1/\delta}$. 
	
	On the other hand, by~\eqref{eq:a_r}, there exists a compact 
	\pmm{}-subspace $\mathcal Z'=(Z',o_Z,\mu'_Z)$ of $\mathcal Z$ such that 
	$\pcball{\mathcal Z}{1/\delta-\delta}\preceq \mathcal Z'$
	and $\cghp(\pcball{\mathcal Y}{1/\delta}, \mathcal Z')<\frac{\delta}{2}$.  By Lemma~\ref{lem:GHP-monotone}, there is a further compact	\pmm{}-subspace $\mathcal Z''$ of $\mathcal Z'$ such that
	$\pcball{\mathcal Z'}{1/(\epsilon+\delta)-\epsilon-\delta}\preceq \mathcal Z''$
	and 
	\[
	\cghp(\mathcal Y', \mathcal Z'')<\frac{\delta}{2}.
	\]
	The triangle inequality for $\cghp$ gives
	\[
	\cghp(\pcball{\mathcal X}{1/(\epsilon+\delta)},\mathcal Z'')\leq \frac{\epsilon}{2} + \frac{\delta}{2}.
	\]
	Since $\pcball{\mathcal Z}{1/(\epsilon+\delta)-\epsilon-\delta}\preceq\mathcal Z''$, \eqref{eq:a_r} implies that $a_{\epsilon+\delta}(\mathcal X,\mathcal Z)<(\epsilon+\delta)/2$. Similarly, one obtains $a_{\epsilon+\delta}(\mathcal Z,\mathcal X)<(\epsilon+\delta)/2$. Therefore, $\ghp(\mathcal X,\mathcal Z)\leq\epsilon+\delta$ and the triangle inequality is proved.
	
	The last step is to prove that $\ghp(\mathcal X,\mathcal Y)=0$ implies that $\mathcal X$ and $\mathcal Y$ are GHP-isometric. Fix $r\geq 0$ and let $0<\epsilon<1$ be arbitrary. Lemma~\ref{lem:GHP-ineqs} implies that $a_{\epsilon}(\mathcal X,\mathcal Y)<\frac{\epsilon}{2}$. Therefore, assuming $r<\frac 1{\epsilon}$, \eqref{eq:a_r} and Lemma~\ref{lem:GHP-monotone} imply 
	that there exists a \pmm-subspace $\mathcal Y_{\epsilon}$ %
	of $\mathcal Y$ such that 
	\[
	\cghp(\pcball{\mathcal X}{r}, \mathcal Y_{\epsilon})<\frac{\epsilon}{2}.
	\]
	By Lemmas~\ref{lem:radius} and~\ref{lem:GHP-subsets}, There is a convergent subsequence of the subspaces under the metric $\cghp$, say $\mathcal Y_{\epsilon_n}\rightarrow \mathcal Y'\preceq \mathcal Y$, where %
	$\epsilon_n\rightarrow 0$. It follows that $\cghp(\pcball{\mathcal X}{r},\mathcal Y')=0$. Since $\cghp$ is a metric on $\mmcstar$,  $\pcball{\mathcal X}{r}$ is GHP-isometric to $\mathcal  Y'$. In particular, Lemma~\ref{lem:radius} implies that $\mathcal Y'\preceq \pcball{\mathcal Y}{r}$. %
	On the other hand, one can similarly find a \pmm-subspace $\mathcal X'$ of $X$ which is GHP-isometric to $\pcball{\mathcal Y}{r}$ and $\mathcal X'\preceq\pcball{\mathcal X}{r}$. These facts imply that $\pcball{\mathcal X}{r}$ and $\pcball{\mathcal Y}{r}$ are themselves GHP-isometric as follows: If $f:\pcball{\mathcal X}{r}\rightarrow \mathcal Y'$ and $g:\pcball{\mathcal Y}{r}\rightarrow \mathcal X'$ are GHP-isometries, then, $g\circ f:\pcball{\mathcal X}{r}\rightarrow \mathcal X'$ is also a GHP-isometry. Compactness of $\pcball{\mathcal X}{r}$, finiteness of the measure on $\pcball{\mathcal X}{r}$ and $\mathcal X'\preceq\pcball{\mathcal X}{r}$ imply that $g\circ f$ is surjective and $\mathcal X'=\pcball{\mathcal X}{r}$.
	
	To prove that $\mathcal X$ is GHP-isometric to $\mathcal Y$, let $C_n$ be the set of GHP-isometries from $\pcball{\mathcal X}{n}$ to $\pcball{\mathcal Y}{n}$ for $n=1,2,\ldots$, which is shown to be non-empty.
	\unwritten{This compactness needs pointwise-continuity.} The topology of uniform convergence makes $C_n$ a compact set. %
	The restriction map $f\mapsto \restrict{f}{\pcball{\mathcal X}{n-1}}$ induces a continuous function from $C_n$ to $C_{n-1}$. %
	Therefore, the generalization of K\"{o}nig's infinity lemma  (Lemma~\ref{lem:infinity}) implies that there is a sequence of GHP-isometries $\rho_n\in C_n$ such that $\rho_{n-1}$ is the restriction of $\rho_n$ to $\pcball{\mathcal X}{n-1}$ for each $n$. Thus, these isometries can be glued together to form a GHP-isometry between $\mathcal X$ and $\mathcal Y$, which proves the claim.
\end{proof}

\begin{remark}
	\label{rem:variant}
	By Lemma~\ref{lem:cadlag}, it is easy to see that
	\begin{equation}
	\label{eq:lengthSpaceGHP}
	\int_0^{\infty} e^{-r}\left(1\wedge \cghp(\pcball{\mathcal X}{r}, \pcball{\mathcal Y}{r}) \right)dr
	\end{equation}
	is well defined for all $\mathcal X,\mathcal Y\in\mmstar$ and defines a semi-metric on $\mmstar$ (such formulas are common in various settings in the literature\del{ some of which are discussed in Sections~\ref{sec:extension} and~\ref{sec:special}}).
	With similar arguments to those in the present section, it can be shown that this is indeed a metric and makes $\mmstar$ a complete separable metric space as well. However, we preferred to use the formulation of Definition~\ref{def:GHP} to avoid the issues regarding non-monotonicity of  $\cghp(\pcball{\mathcal X}{r}, \pcball{\mathcal Y}{r})$ as a function of $r$. In addition, Lemma~\ref{lem:GHP-monotone} enables us to have more quantitative bounds in the arguments. Nevertheless, Theorem~\ref{thm:convergence} below implies that the two metrics generate the same topology.
\end{remark}

\begin{remark}
	\label{rem:noncompactHausdorff}
	Let $Z$ be a metric space, $\mathcal F$ be the set of boundedly-compact subsets of $Z$ and $\mathcal M$ be the set of boundedly-finite Borel measures on $Z$ (up to no equivalence relation).
	By formulas similar to either~\eqref{eq:GHP} or~\eqref{eq:lengthSpaceGHP}, one can extend the Hausdorff metric and the Prokhorov metric to $\mathcal F$ and $\mathcal M$ respectively. This can be done by fixing a point $o\in Z$, letting $\pcball{X}{r}:=X\cap \cball{r}{o}$ for $X\subseteq Z$ and letting $\pcball{\mu}{r}:=\restrict{\mu}{\cball{r}{o}}$ for measures $\mu$ on $Z$ (let $\hausdorff(\emptyset,X):=\infty$ whenever $X\neq\emptyset$).
	By similar arguments, one can show that formulas similar to~\eqref{eq:GHP} or~\eqref{eq:lengthSpaceGHP} give metrics on $\mathcal F$ and $\mathcal M$ respectively. Moreover, if $Z$ is complete and separable, %
	then $\mathcal F$ and $\mathcal M$ are also complete and separable %
	(this can be proved similarly to the results of Subsection~\ref{subsec:polishness} %
	below). In this case, the metrics on $\mathcal F$ and $\mathcal M$ are metrizations of the \textit{Fell topology} and the \textit{vague topology} respectively. The details are skipped for brevity. %
	See Subsection~\ref{subsec:randomMeasure} and~\cite{Kh19generalization} for further discussion.
\end{remark}

\subsection{The Topology of the GHP Metric}
\label{subsec:convergence}

Gromov~\cite{Gr81} has defined a topology on the set of boundedly-compact pointed metric spaces, which is called the \textit{Gromov-Hausdorff topology} in the literature (see also~\cite{bookBBI}). In addition, the \textit{Gromov-Hausdorff-Prokhorov topology} (see~\cite{bookVi10}) is defined  on the set $\mmstar$ of boundedly-compact \pmm{} spaces (it is called the \textit{pointed measured Gromov-Hausdorff topology} in~\cite{bookVi10}). In this subsection, it is shown that the metric $\ghp$ of the present paper is a metrization of the Gromov-Hausdorff-Prokhorov topology. The main result is Theorem~\ref{thm:convergence} which provides criteria for convergence under the metric $\ghp$. The Gromov-Hausdorff topology will be studied in Subsection~\ref{subsec:GH}.

\begin{lemma}
	\label{lem:ghpVScghp}
	Let $\mathcal X, \mathcal Y\in\mmstar$ be \pmm{} spaces.
	\begin{enumerate}[(i)]
		\item \label{lem:ghpVScghp:1} For all $r\geq 0$,
		\[
		\ghp(\mathcal X,\mathcal Y)\leq \frac 1 r \vee 2\cghp(\pcball{\mathcal X}{r},\pcball{\mathcal Y}{r}).
		\]
		\item \label{lem:ghpVScghp:2} If $\mathcal X,\mathcal Y\in\mmcstar$ are compact, then
		\[
		\ghp(\mathcal X, \mathcal Y) \leq 2\cghp(\mathcal X, \mathcal Y).
		\]
		\item \label{lem:ghpVScghp:3} The topology on $\mmcstar$ induced by the metric $\ghp$ is coarser than that of $\cghp$.
	\end{enumerate}
\end{lemma}

\begin{proof}
	\eqref{lem:ghpVScghp:1}.
	Since $\ghp(\mathcal X,\mathcal Y)\leq 1$, we can assume $r\geq 1$ without loss of generality.
	Let $\epsilon>1/ r \vee 2\cghp(\pcball{\mathcal X}{r},\pcball{\mathcal Y}{r})$. It is enough to prove that $\epsilon\geq\ghp(\mathcal X,\mathcal Y)$. This is trivial if $\epsilon\geq 1$. So assume $\epsilon<1$. By letting $\mathcal Y':=\pcball{\mathcal Y}{r}$ in~\eqref{eq:a_r}, one gets that $a(\epsilon,r;\mathcal X,\mathcal Y)\leq \cghp(\pcball{\mathcal X}{r},\pcball{\mathcal Y}{r})<\frac{\epsilon}2$. So, the fact $\frac 1{\epsilon}<r$ and Lemma~\ref{lem:a} imply that $a_{\epsilon}(\mathcal X,\mathcal Y)<\frac{\epsilon}{2}$. Similarly, one gets $a_{\epsilon}(\mathcal Y,\mathcal X)<\frac{\epsilon}{2}$. So~\eqref{eq:GHP} gives that $\ghp(\mathcal X,\mathcal Y)\leq \epsilon$ and the claim is proved.
	
	\eqref{lem:ghpVScghp:2}. The claim is implied by part~\eqref{lem:ghpVScghp:1} by letting $r$ large enough such that $\pcball{\mathcal X}{r}=\mathcal X$, $\pcball{\mathcal Y}{r}=\mathcal Y$ and $\frac 1 r\leq 2\cghp(\mathcal X,\mathcal Y)$.

	\eqref{lem:ghpVScghp:3}.
	By the previous part, any convergent sequence under $\cghp$ is also convergent under $\ghp$. This implies the second claim.
\end{proof}

\begin{remark} 
	In fact, the topology of the metric $\ghp$ on $\mmcstar$ is strictly coarser than that of $\cghp$
	since having $\ghp(\mathcal X_n,\mathcal X)\to 0$ does not imply $\cghp(\mathcal X_n,\mathcal X)\to 0$; e.g., when $\mathcal X_n:=\{0,n\}$ and $\mathcal X:=\{0\}$ endowed with the Euclidean metric and the counting measure (in general, adding the assumption $\sup \diam(\mathcal X_n)<\infty$ is sufficient for convergence under $\cghp$).
	This is similar to the fact that the vague topology on the set of measures on a given non-compact metric space is strictly coarser than the weak topology (see e.g., \cite{bookKa17randommeasures}). A similar property holds for the set of compact subsets of a given non-compact metric space. %
\end{remark}

\begin{theorem}[Convergence]
	\label{thm:convergence}
	Let $\mathcal X$ and $(\mathcal X_n)_{n\geq 0}$ be boundedly compact \pmm{} spaces. Then the following are equivalent:
	\begin{enumerate}[(i)]
		\item \label{thm:convergence:1} $\mathcal X_n\to \mathcal X$ in the metric $\ghp$.
		\item \label{thm:convergence:2} For every $r>0$ and $\epsilon>0$, for large enough $n$, there exists a compact \pmm-subspace $\mathcal X'_n$ of $\mathcal X$ such that $\pcball{\mathcal X}{r-\epsilon} \preceq\mathcal X'_n\preceq X$  and $\cghp(\pcball{\mathcal X}{r}_n, \mathcal X'_n)<\epsilon$.
		
		\item \label{thm:convergence:3} For every $r>0$ and $\epsilon>0$, for large enough $n$, there exist
		compact \pmm-subspaces of $\mathcal X$ and $\mathcal X_n$ with $\cghp$-distance less than $\epsilon$ such that they contain (as \pmm-subspaces) the balls of radii $r$ centered at the roots of $\mathcal X$ and $\mathcal X_n$ respectively.

		\item \label{thm:convergence:4} 
		For every continuity radius $r$ of $\mathcal X$ (Definition~\ref{def:continuityRadius}),
		one has $\pcball{\mathcal X}{r}_n\to\pcball{\mathcal X}{r}$ in the metric $\cghp$ as $n\to\infty$. 
		
		\item \label{thm:convergence:5}
		There exists an unbounded set $I\subseteq \mathbb R^{\geq 0}$ such that for each $r\in I$,  
		one has $\pcball{\mathcal X}{r}_n\to\pcball{\mathcal X}{r}$ in the metric $\cghp$ as $n\to\infty$. 
		\item \label{thm:convergence:6}
		$\lim_{n\to\infty} \int_0^{\infty} e^{-r}\left(1\wedge \cghp(\pcball{\mathcal X}{r}_n, \pcball{\mathcal X}{r}) \right)dr = 0$.
	\end{enumerate}
\end{theorem}
\begin{proof}
	\eqref{thm:convergence:1}$\Rightarrow$\eqref{thm:convergence:2}.
	Assume $\mathcal X_n\to \mathcal X$. Let $r>0$ and $\epsilon>0$ be given. One may assume $\epsilon<\frac 1 r$ without loss of generality. For large enough $n$, one has $\ghp(\mathcal X_n, \mathcal X)<\epsilon$. If so, Lemma~\ref{lem:GHP-ineqs} imply that $a_{\epsilon}(\mathcal X_n,\mathcal X)<\frac{\epsilon}{2}$. So Lemma~\ref{lem:a} gives $a(\epsilon,r;\mathcal X_n,\mathcal X)<\frac{\epsilon}{2}$. Now, the claim is implied by~\eqref{eq:a_r}.
	
	\eqref{thm:convergence:2}$\Rightarrow$\eqref{thm:convergence:3}. The claim of part~\eqref{thm:convergence:3} is directly implied from part~\eqref{thm:convergence:2} by replacing $r$ with $r+\epsilon$.
	
	\eqref{thm:convergence:3}$\Rightarrow$\eqref{thm:convergence:1}. Let $\epsilon>0$ be arbitrary and $r=1/(2\epsilon)$. Assume $n$ is large enough such that there exist compact \pmm-subspaces $\pcball{\mathcal X}{r}\preceq \mathcal X'\preceq\mathcal X$ and $\pcball{\mathcal X}{r}_n\preceq \mathcal X'_n\preceq\mathcal X_n$ such that $\cghp(\mathcal X',\mathcal X'_n)<\epsilon$. By Lemma~\ref{lem:GHP-monotone}, there exists a compact \pmm-subspace $\pcball{\mathcal X}{r-2\epsilon}\preceq \mathcal X''\preceq\mathcal X$ such that $\ghp(\pcball{\mathcal X}{r}_n, \mathcal X'')<\epsilon$. This implies that $a(2\epsilon,r;\mathcal X_n,\mathcal X)<\epsilon$, hence, $a_{2\epsilon}(\mathcal X_n,\mathcal X)<\epsilon$. Similarly, $a_{2\epsilon}(\mathcal X,\mathcal X_n)<\epsilon$, which implies that $\ghp(\mathcal X_n,\mathcal X)\leq 2\epsilon$. This proves that $\mathcal X_n\to\mathcal X$.
	
	\eqref{thm:convergence:2}$\Rightarrow$\eqref{thm:convergence:4}. Let $\mathcal X_n=:(X_n,o_n,\mu_n)$, $\mathcal X=:(X,o,\mu)$ and $r$ be a continuity radius for $\mathcal X$. 
	Let $\epsilon>0$ be arbitrary. The assumption on $r$ implies that there exists $\delta>0$ such that
	\begin{eqnarray*}
		\label{eq:thm:convergence:1}
		\hausdorff(\cball{r+\delta}{o} ,\cball{r-\delta}{o})&\leq& \epsilon,\\
		\label{eq:thm:convergence:2}
		\prokhorov(\restrict{\mu}{\cball{r+\delta}{o}}, \restrict{\mu}{\cball{r-\delta}{o}})&\leq&\epsilon.
	\end{eqnarray*}
	Part~\eqref{thm:convergence:2}, which is assumed, implies that for large enough $n$, there exists a compact \pmm-subspace $\mathcal Y_n$ of $\mathcal X$ such that $\pcball{\mathcal X}{r-\delta} \preceq\mathcal Y_n\preceq X$  and $\cghp(\pcball{\mathcal X}{r}_n, \mathcal Y_n)<\delta/2$. 
	The latter and Lemma~\ref{lem:radius} imply that $\mathcal Y_n\preceq\pcball{\mathcal X}{r+\delta}$. Now, $\mathcal Y_n$ and $\pcball{\mathcal X}{r}$ both contain (ass \pmm-subspaces) $\pcball{\mathcal X}{r-\delta}$ and are contained in $\pcball{\mathcal X}{r+\delta}$. By using the definitions~\eqref{eq:hausdorff} and~\eqref{eq:dProkhorov} of the Hausdorff and the Prokhorov metrics directly, one can deduce that $\hp(\mathcal Y_n,\pcball{\mathcal X}{r})\leq \hp(\pcball{\mathcal X}{r+\delta},\pcball{\mathcal X}{r-\delta})\leq \epsilon$. So~\eqref{eq:hp-ineq} implies that $\cghp(\mathcal Y_n,\pcball{\mathcal X}{r})\leq \epsilon$. 
	Finally, the triangle inequality gives $\cghp(\pcball{\mathcal X}{r}_n,\pcball{\mathcal X}{r})<\epsilon+\frac \delta 2$. Since $\epsilon,\delta$ are arbitrarily small, this implies that $\pcball{\mathcal X}{r}_n\to\pcball{\mathcal X}{r}$ and the claim is proved.

	\eqref{thm:convergence:4}$\Rightarrow$\eqref{thm:convergence:5}. The claim is implied by Lemma~\ref{lem:continuityRadius}.

	\eqref{thm:convergence:5}$\Rightarrow$\eqref{thm:convergence:1}. 
	The claim is easily implied by part~\eqref{lem:ghpVScghp:1} of Lemma~\ref{lem:ghpVScghp} and is left to the reader.
	
	\eqref{thm:convergence:4}$\Rightarrow$\eqref{thm:convergence:6}. By Lemma~\ref{lem:cadlag2}, the integrand is a c\`adl\`ag function of $r$, and hence, measurable. Since $\mathcal X$ has countably many discontinuity radii (Lemma~\ref{lem:continuityRadius}), the claim follows by Lebesgue's dominated convergence theorem.
	
	\eqref{thm:convergence:6}$\Rightarrow$\eqref{thm:convergence:4}. To prove this part, some care is needed since the converse of the dominated convergence theorem does not hold in general, and hence, the above arguments do not work. %
	Let $r$ be a continuity radius of $\mathcal X$ and $0<\epsilon<1$. By Definition~\ref{def:continuityRadius}, there exists $\delta>0$ such that $\delta<\epsilon/2$ and
	\[
	\hp(\pcball{\mathcal X}{r+2\delta}, \pcball{\mathcal X}{r-2\delta})< \frac{\epsilon}{2}.
	\]
	Let $\gamma_n(s):=\cghp(\pcball{\mathcal X}{s}_n,\pcball{\mathcal X}{s})$.
	By~\eqref{thm:convergence:6}, there exists $N$ such that for all $n\geq N$,
	\begin{equation}
	\label{eq:thm:convergence:3}
	\int_0^{\infty} e^{-s}\left(1\wedge \gamma_n(s)\right)ds <\delta e^{-r}.
	\end{equation}
	To prove the claim, it is enough to show that for all $n\geq N$, one has 
	$\gamma_n(r)\leq\epsilon$.
	Let $n\geq N$ be arbitrary. %
	First, assume that there exists $s>r$ such that 
	$\gamma_n(s)\leq \delta$.
	By Lemmas~\ref{lem:GHP-monotone} and~\ref{lem:radius}, there exists a compact \pmm{}-subspace $\mathcal X'\preceq \mathcal X$ such that $\cghp(\pcball{\mathcal X}{r}_n,\mathcal X')\leq \delta$ and  $\pcball{\mathcal X}{r-2\delta}\preceq\mathcal X'\preceq \pcball{\mathcal X}{r+2\delta}$. It can be seen that the latter implies that
	$$\hp(\pcball{\mathcal X}{r},\mathcal X')\leq \hp(\pcball{\mathcal X}{r+2\delta}, \pcball{\mathcal X}{r-2\delta})<\frac{\epsilon}{2}.$$
	The triangle inequality for $\cghp$ gives that
	\[
	\gamma_n(r)=\cghp(\pcball{\mathcal X}{r}_n,\pcball{\mathcal X}{r}) \leq \cghp(\pcball{\mathcal X}{r}_n,\mathcal X') + \hp(\mathcal X', \pcball{\mathcal X}{r})<\epsilon.
	\]
	So the claim is proved in this case. Second, assume that for all $s>r$, one has $\gamma(s)>\delta$. This gives that $\int_0^{\infty}e^{-s}(1\wedge \gamma(s))ds \geq \delta e^{-r}$. This contradicts~\eqref{eq:thm:convergence:3}. So the claim is proved.
\end{proof}

It is known that convergence under the metric $\cghp$ can be expressed using \textit{approximate GHP-isometries} (see e.g., page~767 of~\cite{bookVi10} and Corollary~7.3.28 of~\cite{bookBBI}). This is expressed in the following lemma, whose proof is skipped. %

An \defstyle{$\epsilon$-isometry} (see e.g., \cite{bookBBI}) between metric spaces $X$ and $Y$ is a function $f:X\to Y$ such that $\sup\{\norm{d(x_1,x_2)-d(f(x_1),f(x_2))}:x_1,x_2\in X \}\leq\epsilon$ and for every $y\in Y$, there exists $x\in X$ such that $d(y,f(x))\leq \epsilon$.

\begin{lemma}
	\label{lem:GHPtopology}
	Let $\mathcal X=(X,o,\mu)$ and $\mathcal X_n=(X_n,o_n,\mu_n)$ be compact \pmm-spaces ($n=1,2,\ldots$). Then $\mathcal X_n\to\mathcal X$ in the metric $\cghp$ if and only if for every $\epsilon>0$, for large enough $n$, there exists a measurable $\epsilon$-isometry $f:X_n\to X$ such that $f(o_n)=o$ and $\prokhorov(f_*\mu_n,\mu)<\epsilon$.
\end{lemma}

In fact, one can prove a quantitative form of this lemma that relates the existence of such $f$ to the value of $\cghp(\mathcal X_n,\mathcal X)$ (similarly to Equation~(27.3) of~\cite{bookVi10}).

The notion of approximate GHP-isometries is also used in~\cite{bookBBI} and Definition~27.30 of~\cite{bookVi10} to define convergence of boundedly-compact \pms{} spaces and \pmm{} spaces as follows: %
$(X_n,o_n,\mu_n)$ tends to $(X,o,\mu)$ when there exist sequences $r_k\to\infty$ and $\epsilon_k\to 0$ and measurable $\epsilon_k$-isometries $f_k:\cball{r_k}{o_k}\to\cball{r_k}{o}$ such that $f_{k*}\mu_k$ tends to $\mu$ in the weak-$*$ topology (convergence against compactly supported continuous functions). %
By part~\eqref{thm:convergence:5} of Theorem~\ref{thm:convergence}, the reader can verify the following.

\begin{theorem}
	\label{thm:metrization}
	The metric $\ghp$ is a metrization of the %
	Gromov-Hausdorff-Prokhorov topology (Definition~27.30 of~\cite{bookVi10}).
\end{theorem}

See also Theorem~\ref{thm:metrizationGH} for a version of this result for the Gromov-Hausdorff topology.

%
%
%

%
%
%
%
%
%
%
%
%
%
%
%
%
%
%
%
%
%
%
%
%
%
%
%
%
%
%
%
%
%

\subsection{Completeness, Separability and Pre-Compactness}
\label{subsec:polishness}
The following two theorems are the main results of this subsection. Recall that a Polish space is a topological space which is homeomorphic to a complete separable metric space.

\begin{theorem}
	\label{thm:GHPcomplete}
	Under the GHP metric, $\mmstar$ is a complete separable metric space.
\end{theorem}

The proof of Theorem~\ref{thm:GHPcomplete} is postponed to after proving Theorem~\ref{thm:GHP-precompactness}.

Recall that a subset $S$ of a metric space $X$ is \textit{relatively compact} (or \textit{pre-compact}) when every sequence in $S$ has a subsequence which is convergent in $X$; i.e., the closure of $S$ in $X$ is compact. The following gives a pre-compactness criteria for the GHP metric. %

\begin{theorem}[Pre-compactness]
	\label{thm:GHP-precompactness}
	A subset $\mathcal C\subseteq \mmstar$ is relatively compact under the GHP metric if and only if for each $r\geq 0$, the set of (equivalence classes of the)
	balls $\mathcal C_r:=\{\pcball{\mathcal X}{r}: \mathcal X\in \mathcal C \}$ is relatively compact under the metric $\cghp$.
\end{theorem}

For a pre-compactness criteria for the metric $\cghp$, see Theorem~2.6 of~\cite{AbDeHo13}. \del{See also Remark~\ref{rem:functor-precompact} below.}

\begin{proof}[Proof of Theorem~\ref{thm:GHP-precompactness}] ($\Rightarrow$).
	First, assume $\mathcal C$ is pre-compact, $r\geq 0$ and $(\mathcal X_n)_n$ is a sequence in $\mathcal C$. We will prove that the sequence $(\pcball{\mathcal X}{r}_n)_n$ has a convergent subsequence, which proves that $\mathcal C_r$ is pre-compact. By pre-compactness of $\mathcal C$, one finds a convergent subsequence of $\mathcal X_n$. So, one may assume $\mathcal X_n\rightarrow \mathcal Y$ under the metric $\ghp$ from the beginning without loss of generality.
	Choose $\epsilon_n>\ghp(\mathcal X_n,\mathcal Y)$ such that  $\epsilon_n\rightarrow 0$. We can assume $\epsilon_n<1$ for all $n$ without loss of generality. Lemma~\ref{lem:GHP-ineqs} implies that $a_{\epsilon_n}(\mathcal X_n,\mathcal Y)< \frac 12\epsilon_n$. So, Lemma~\ref{lem:a}  implies that $\delta_n:=a(1/r, r;\mathcal X_n,\mathcal Y)\rightarrow 0$.
	By the definition of $a$ in~\eqref{eq:a_r} and Lemma~\ref{lem:a-inf}, one finds a \pmm-subspace $\mathcal Y_n$ of $\mathcal Y$ such that 
	\begin{equation}
	\label{eq:thm:GHP-precompactness:1}
	\cghp(\pcball{\mathcal X}{r}_n, \mathcal Y_n)\leq\delta_n.
	\end{equation}
	Lemma~\ref{lem:radius} gives $\mathcal Y_n \preceq \pcball{\mathcal Y}{r+2\delta_n}$. So, by Lemma~\ref{lem:GHP-subsets}, one can find a convergent subsequence of the subspaces $\mathcal Y_n$ under the metric $\cghp$, say tending to $\mathcal Y'\preceq \mathcal Y$. By passing to this subsequence, one may assume $\mathcal  Y_n\rightarrow \mathcal Y'$ from the beginning. Now, \eqref{eq:thm:GHP-precompactness:1} implies that $\cghp(\pcball{\mathcal X}{r}_n, \mathcal Y')\rightarrow 0$, which proves the claim (it should be noted that the limit $\mathcal Y'$ satisfies $\poball{\mathcal Y}{r}\preceq \mathcal Y'\preceq \pcball{\mathcal Y}{r}$, but is not necessarily equal to $\pcball{\mathcal Y}{r}$).

	($\Leftarrow$).
	Conversely assume $\mathcal C_r$ is pre-compact for every $r\geq 0$. Let $(\mathcal X_n)_n$ be a sequence in $\mathcal C$. The claim is that it has a convergent subsequence under the metric $\ghp$. %
	For each given $m\in \mathbb N$, by pre-compactness of $\mathcal C_m$, one finds a subsequence of $(\pcball{\mathcal X}{m}_n)_n$ that is convergent in the $\cghp$ metric. By a diagonal argument, one finds a subsequence $n_1<n_2<\ldots$ such that for every $m\in\mathbb N$, the sequence $\pcball{\mathcal X}{m}_{n_i}$ is convergent as $i\rightarrow \infty$. By passing to this subsequence, we may assume from the beginning that $\pcball{\mathcal X}{m}_{n}$ is convergent as $n\rightarrow\infty$, say, to $\mathcal Y_m$, for each $m\in\mathbb N$ (in the metric $\cghp$); i.e.,  
	\[
	\forall m\in\mathbb N:\pcball{\mathcal X}{m}_{n}\to\mathcal Y_m.
	\]
	The next step is to show that these limiting spaces $\mathcal Y_m$ can be glued together to form a \pmm{} space.
	Let $1<m\in\mathbb N$ be given. For each $n$, Lemma~\ref{lem:GHP-monotone} implies that there is a \pmm-subspace $\mathcal Z_{m,n}$ of $\mathcal Y_m$ %
	such that $\cghp(\pcball{\mathcal X}{m-1}_n, \mathcal Z_{m,n})\leq \cghp(\pcball{\mathcal X}{m}_n,\mathcal Y_m)$. This implies that $\cghp(\pcball{\mathcal X}{m-1}_n, \mathcal Z_{m,n})\rightarrow 0$ as $n\rightarrow\infty$. By Lemma~\ref{lem:GHP-subsets}, the sequence $(\mathcal Z_{m,n})_n$ has a convergent subsequence {in the metric $\cghp$}, say, tending to $\mathcal Z_m\preceq \mathcal Y_m$. Therefore, $\cghp(\pcball{\mathcal X}{m-1}_n, \mathcal Z_{m})$ tends to zero along the subsequence. On the other hand, the definition of $\mathcal Y_{m-1}$ implies that  $\pcball{\mathcal X}{m-1}_n\rightarrow \mathcal Y_{m-1}$ as $n\rightarrow\infty$. Thus, $\cghp(\mathcal Y_{m-1}, \mathcal Z_m)=0$; i.e., $\mathcal Y_{m-1}$ is GHP-isometric to $\mathcal Z_m$ which is a \pmm-subspace of $\mathcal Y_m$. This shows that $\mathcal Y_m$'s can be paste together to form a \pmm{} space which is denoted by $\mathcal Y$. So, from the beginning, we may assume $\mathcal Y_m$ is a \pmm-subspace of $\mathcal Y$ for each $m$. %
	
	In the next step, it will be shown that $\mathcal Y$ is boundedly-compact. \unwritten{In the definition of $a$, the condition $Y'\supseteq ...$ is not necessary in the following argument.}
	The above application of Lemma~\ref{lem:GHP-monotone} also implies that $\mathcal Z_{m,n}$ contains a large ball in $\mathcal Y_m$. More precisely, $\pcball{\mathcal Y}{m-1-\delta_n}_m\preceq \mathcal Z_{m,n}$ for some $\delta_n>0$ that tends to zero.  
	By letting $n$ tend to infinity, we get
	$\poball{\mathcal Y}{m-1}_m\preceq \mathcal Y_{m-1}$ 
	(assuming $\mathcal Y_{m-1}$ is a \pmm-subspace of $\mathcal Y_m$ as above). By an induction, one obtains that $\poball{\mathcal Y}{m-1}_{m'}\preceq\mathcal Y_{m-1}$ for every $m'\geq m$. Now, the definition of $\mathcal Y$ implies that	
	$\poball{\mathcal Y}{m} \preceq \mathcal Y_m$ (note also that $\mathcal Y_m\preceq \pcball{\mathcal Y}{m}$). %
	This implies that %
	$\mathcal Y$ is boundedly-compact. 
	
	The final step is  to show that $\mathcal X_n\rightarrow \mathcal Y$ in the metric $\ghp$.
	Fix $\epsilon>0$ and let $m>1/\epsilon$ be arbitrary. 
	Equation \eqref{eq:a_r} and	$\poball{\mathcal Y}{m}\preceq \mathcal Y_m$ imply that $a(\epsilon,m; \mathcal X_n,\mathcal Y)\leq  \cghp(\pcball{\mathcal X}{m}_n, \mathcal Y_m)$. By using Lemma~\ref{lem:a} and the fact that $\cghp(\pcball{\mathcal X}{m}_n, \mathcal Y_m)$ tends to zero as $n\rightarrow 0$ one can show that $a_{\epsilon}(\mathcal X_n,\mathcal Y)\rightarrow 0$. On the other hand, since $\pcball{\mathcal Y}{m-1}\preceq\poball{\mathcal Y}{m}\preceq\mathcal Y_m$, one can use Lemma~\ref{lem:GHP-monotone} and show that $a(\epsilon,m-1;\mathcal Y,\mathcal X_n)\rightarrow 0$ as $n\rightarrow \infty$. By similar arguments, one can show that $a_{\epsilon}(\mathcal Y,\mathcal X_n)\to 0$.
	This implies that $\ghp(\mathcal X_n,\mathcal Y)<\epsilon$ for large enough $n$ (see Definition~\ref{def:GHP}). Since $\epsilon$ is arbitrary, one gets $\mathcal X_n\rightarrow \mathcal Y$ and the claim is proved.
\end{proof}

\begin{proof}[Proof of Theorem~\ref{thm:GHPcomplete}]
	The definition of the GHP metric directly implies that $$\ghp(\mathcal X, \pcball{\mathcal X}{r})\leq \frac 1 r$$ for every $\mathcal X\in \mmstar$ and $r>0$. 
	Hence, $ \pcball{\mathcal X}{r}\rightarrow \mathcal X$ as $r\rightarrow \infty$. So, the subset $\mmcstar\subseteq \mmstar$ formed by compact spaces is dense. As noted in Subsection~\ref{subsec:GHPcompact}, $\mmcstar$ is separable under the metric $\cghp$. %
	Lemma~\ref{lem:ghpVScghp} implies that $\mmcstar$ is separable under $\ghp$ as well. One obtains that $\mmstar$ is separable.
	
	For proving completeness, assume $(\mathcal X_n)_n$ is a Cauchy sequence in $\mmstar$ under the metric $\ghp$. Below, we will show that this sequence is pre-compact. This proves that there exists a convergent subsequence. Being Cauchy implies convergence of the whole sequence and the claim is proved.
	By Theorem~\ref{thm:GHP-precompactness}, to show pre-compactness of the sequence, it is enough to prove that for a given $r\geq 0$, the sequence of balls $(\pcball{\mathcal X}{r}_n)_n$ is pre-compact 
	under the metric $\cghp$.
	
	Let $0<\epsilon<\frac 1 r$. There exists $m$ such that for all $n>m$,  $\ghp(\mathcal X_n,\mathcal X_m)<\epsilon$. By Lemmas~\ref{lem:GHP-ineqs} and~\ref{lem:a}, one gets $a(\epsilon,r;\mathcal X_n,\mathcal X_m)< \frac{\epsilon}{2}$. Therefore, there exists a compact \pmm-subspace $\mathcal Z_{m,n}$ of $\mathcal X_m$ such that $\cghp(\pcball{\mathcal X}{r}_n, \mathcal Z_{m,n})\leq \frac{\epsilon}{2}$. Lemma~\ref{lem:radius} gives that $\mathcal Z_{m,n}\preceq \pcball{\mathcal X}{r+\epsilon}_m$. So, by Lemma~\ref{lem:GHP-subsets}, the sequence $(\mathcal Z_{m,n})_n$ has a convergent subsequence under the metric $\cghp$, say, tending to  $\mathcal Z_m\preceq\mathcal X_m$.
	Therefore, one finds a subsequence of the balls $(\pcball{\mathcal X}{r}_n)_{n>m}$ such that $\cghp(\pcball{\mathcal X}{r}_n, \mathcal Z_m)<\epsilon$ on the subsequence. Hence, any two elements of the subsequence have distance less than $2\epsilon$. By doing this for different values of $\epsilon$ iteratively; e.g., for $\epsilon = \frac 1{2r}, \frac 1{3r}, \ldots$, and by a diagonal argument, one finds a sequence $n_1,n_2,\ldots$ such that $(\pcball{\mathcal X}{r}_{n_i})_i$ is a Cauchy sequence under the metric $\cghp$. Therefore, by completeness of the metric $\cghp$ (see Subsection~\ref{subsec:GHPcompact}), this sequence is convergent. So, by the arguments of the previous paragraph, the claim is proved.
\end{proof}

\subsection{Random \pmm{} Spaces and Weak Convergence}
\label{subsec:weak}
Theorem~\ref{thm:GHPcomplete} shows that the space $\mmstar$, equipped with the GHP metric $\ghp$, is a Polish space. This enables one to define a \defstyle{random \pmm{} space} $\bs{\mathcal X}$ as a random element in $\mmstar$ and the probability space will be standard. The distribution of $\bs{\mathcal X}$ is the probability measure $\mu$ on $\mmstar$ defined by $\mu(A):=\myprob{\bs{\mathcal X}\in A}$. In this subsection, weak convergence of random \pmm{} spaces are studied.

Let $\bs{\mathcal X}_1,\bs{\mathcal X}_2,\cdots$ and $\bs{\mathcal X}$ be random \pmm{} spaces. Let $\mu_n$ (resp. $\mu$) be the distribution of $\bs{\mathcal X}_n$ (resp. $\bs{\mathcal X}$). Prokhorov's theorem~\cite{Pr56} implies that $\bs {\mathcal X}_n$ converges weakly to $\mathcal X$ if and only if $\prokhorov(\mu_n,\mu)\to 0$, where $\prokhorov$ is the Prokhorov metric corresponding to the metric $\ghp$. 

In the following, let $\prokhorov^c$ be the Prokhorov metric corresponding to the metric $\cghp$ on $\mmcstar$. For given $r\geq 0$, it can be seen that the projection $\mathcal X\mapsto \pcball{\mathcal X}{r}$ from $\mmstar$ to $\mmcstar$ is measurable. So the ball $\pcball{\bs{\mathcal X}}{r}$ is well defined as a random element of $\mmcstar$. Let $\mu^{(r)}$ be the distribution of $\pcball{\bs{\mathcal X}}{r}$.

\begin{lemma}
	\label{lem:weak}
	Let $\bs{\mathcal X}$ and $\bs{\mathcal Y}$ be random \pmm{} spaces with distributions $\mu$ and $\nu$ respectively.
	\begin{enumerate}[(i)]
		\item \label{lem:weak:1} For every $r\geq 0$,
		\[
		\prokhorov(\mu,\nu)\leq \frac 1  r \vee 2\prokhorov^c(\mu^{(r)},\nu^{(r)}).
		\]
		\item \label{lem:weak:2} If $\bs{\mathcal X,\mathcal Y}$ are compact a.s. (i.e., are random elements of $\mmcstar$), then
		\[
		\prokhorov(\mu,\nu)\leq 2\prokhorov^c(\mu,\nu).
		\]
	\end{enumerate}
\end{lemma}

\begin{proof}
	\eqref{lem:weak:1}. 
	Let $\epsilon> \frac 1  r \vee 2\prokhorov^c(\mu^{(r)},\nu^{(r)})$. The goal is to prove that $\epsilon\geq \prokhorov(\mu,\nu)$. One can assume $\epsilon< 1< r$ without loss of generality. By Strassen's theorem (Corollary~\ref{cor:strassen}), there exists a coupling of $\bs{\mathcal X,\mathcal Y}$ such that
	\[
	\myprob{\cghp(\pcball{\bs{\mathcal X}}{r},\pcball{\bs{\mathcal Y}}{r})>\frac{\epsilon}{2}}\leq\frac{\epsilon}{2}.
	\]
	So part~\eqref{lem:ghpVScghp:1} of Lemma~\ref{lem:ghpVScghp} and the assumption $\epsilon>\frac 1 r$ give
	\[
	\myprob{\ghp({\bs{\mathcal X}},{\bs{\mathcal Y}})>{\epsilon}}\leq\frac{\epsilon}{2}\leq \epsilon.
	\]
	So the converse of Strassen's theorem (see Theorem~\ref{thm:strassen}) implies that $\prokhorov(\mu,\nu)\leq \epsilon$ and the claim is proved.
	
	\eqref{lem:weak:2}. 
	Let $\epsilon>0$ be arbitrary. One can choose $r>\frac 1{\epsilon}$  large enough  such that
	$\myprob{r>\diam(\bs{\mathcal X})}<\epsilon$. This implies that $\myprob{\pcball{\bs{\mathcal X}}{r}\neq\bs{\mathcal X}}<\epsilon$. Choose $r$ such that the same holds for $\bs{\mathcal Y}$. So the converse of Strassen's theorem implies that $\prokhorov^c(\mu,\mu^{(r)})\vee \prokhorov^c(\nu,\nu^{(r)})\leq \epsilon$. Now, part~\eqref{lem:weak:1} and the triangle inequality give 
	$$\prokhorov(\mu,\nu)\leq \frac 1 r \vee 2(\prokhorov^c(\mu,\nu)+2\epsilon)\leq 2\prokhorov^c(\mu,\nu)+5\epsilon.$$
	Since $\epsilon$ is arbitrary, the claim is proved. 
\end{proof}

The following result relates weak convergence in $\mmstar$ to that in $\mmcstar$.  Below, a number $r>0$ is called a \defstyle{continuity radius of $\mu$} if it is a continuity radius (Definition~\ref{def:continuityRadius}) of $\bs{\mathcal X}$ almost surely. 

\begin{theorem}[Weak Convergence]
	\label{thm:weak}
	Let $\bs{\mathcal X}_1,\bs{\mathcal X}_2,\cdots$ and $\bs{\mathcal X}$ be random \pmm{} spaces with distributions $\mu_1,\mu_2,\ldots$ and $\mu$ respectively. Then the following are equivalent.
	\begin{enumerate}[(i)]
		\item \label{thm:weak:1} $\bs{\mathcal X}_n\Rightarrow \bs{\mathcal X}$ weakly; i.e., $\prokhorov(\mu_n,\mu)\to 0$.
		\item \label{thm:weak:2} For every continuity radius $r$ of $\mu$,  $\pcball{\bs{\mathcal X}}{r}_n\Rightarrow \pcball{\bs{\mathcal X}}{r}$ weakly as random elements of $\mmcstar$; i.e., $\prokhorov^c(\mu_n^{(r)},\mu^{(r)})\to 0$.
		\item \label{thm:weak:3} There exists an unbounded set $I\subseteq \mathbb R^{\geq 0}$ such that $\pcball{\bs{\mathcal X}}{r}_n\Rightarrow \pcball{\bs{\mathcal X}}{r}$ weakly for every $r\in I$.
	\end{enumerate}
\end{theorem}

\begin{proof}
	\eqref{thm:weak:1}$\Rightarrow$\eqref{thm:weak:2}. Let $r$ be a continuity radius of $\mu$. 
	Therefore, as $\delta\to 0$, $\hp(\pcball{\bs{\mathcal X}}{r+\delta},\pcball{\bs{\mathcal X}}{r-\delta})\to 0$ a.s. (see~\eqref{eq:HP}). So, by fixing $\epsilon>0$ arbitrarily, the following holds for small enough $\delta$. %
	\[
	\myprob{\hp(\pcball{\bs{\mathcal X}}{r+\delta},\pcball{\bs{\mathcal X}}{r-\delta})>\epsilon}<\epsilon.
	\]
	Assume that $0<\delta<r \wedge \frac 1r$.
	The assumption of~\eqref{thm:weak:1} implies that for large enough $n$,  $\prokhorov(\mu_n,\mu)<\frac{\delta}{2}$. Fix such $n$. By Strassen's theorem (Corollary~\ref{cor:strassen}), there exists a coupling of $\bs{\mathcal X}_n$ and $\bs{\mathcal X}$ such that $\myprob{\ghp(\bs{\mathcal X}_n,\bs{\mathcal X})>\frac{\delta}{2}}\leq\frac{\delta}{2}$. Similarly to the proof of \eqref{thm:convergence:2}$\Rightarrow$\eqref{thm:convergence:4} of Theorem~\ref{thm:convergence}, by using Lemma~\ref{lem:GHP-monotone} and the above inequality, one can deduce that 
	\[
	\myprob{\cghp(\pcball{\bs{\mathcal X}}{r}_n, \pcball{\bs{\mathcal X}}{r})>\epsilon+\frac{\delta}{2}}<\epsilon+\frac{\delta}{2}.
	\]
	Now, the converse of Strassen's theorem shows that $\prokhorov^c(\mu_n^{(r)}, \mu^{(r)})\leq \epsilon+\frac{\delta}{2}$. Since the RHS is arbitrarily small, the claim is proved.
	
	\eqref{thm:weak:2}$\Rightarrow$\eqref{thm:weak:3}. 
	By Lemma~\ref{lem:continuityRadius} and Fubini's theorem, one can show that the set of discontinuity radii of $\mu$ has zero Lebesgue measure. This implies the claim.
	
	\eqref{thm:weak:3}$\Rightarrow$\eqref{thm:weak:1}. 
	The claim is implied by part~\eqref{lem:weak:1} of Lemma~\ref{lem:weak} and is left to the reader.
\end{proof}

\begin{remark}
	\unwritten{\later: Delete this?}
	Part~\eqref{thm:weak:2} of Theorem~\ref{thm:weak} is similar to the convergence of finite dimensional distributions in stochastic processes (note that one can identify a random \pmm{} space $\bs{\mathcal X}$ with the stochastic process $t\mapsto \pcball{\bs{\mathcal X}}{t}$ in $\mmcstar$), but a stronger result holds: Convergence of one-dimensional marginal distributions, only for the set of continuity radii, is enough for the convergence of the whole process in this case. This is due to the monotonicity in Lemma~\ref{lem:a}. See also Subsection~\ref{subsec:skorokhod} below.
\end{remark}

\section{Special Cases and Connections to Other Notions}
\label{sec:special}

This section discusses some notions in the literature which are special cases of, or connected to, the Gromov-Hausdorff-Prokhorov metric defined in this paper.

\subsection{A Metrization of the Gromov-Hausdorff Convergence}
\label{subsec:GH}

Here, it is shown that the setting of Section~\ref{sec:ghp} can be used to extend the Gromov-Hausdorff metric to the boundedly-compact case. Also, it is shown that this gives a metrization of the \textit{Gromov-Hausdorff topology} on the set $\mstar$ of boundedly-compact pointed metric spaces.  
In addition, it is shown that $\mstar$ is a Polish space, which enables one to define \textit{random boundedly-compact pointed metric spaces} %
(see Subsections~\ref{subsec:length} and~\ref{subsec:discrete} below for metrics on  specific subsets of $\mstar$).

First, the Gromov-Hausdorff metric is recalled in the compact case (see~\cite{Gr81} or~\cite{bookBBI}). The original definition~\eqref{eq:GH} is for non-pointed spaces, but we recall the pointed version since it will be used later.
Let $\mathcal X=(X,o_X)$ and $\mathcal Y=(Y,o_Y	)$ be compact pointed metric spaces. The \defstyle{Gromov-Hausdorff distance} $\cgh(\mathcal X,\mathcal Y)$ of $\mathcal X$ and $\mathcal Y$ is defined similar to the metric $\cghp$ of Subsection~\ref{subsec:GHPcompact} by deleting the last term in~\eqref{eq:GHPcompact}; or equivalently, by letting $\mu_X$ and $\mu_Y$ be the zero measures in~\eqref{eq:GHPcompact}. It is known that $\cgh$ is a metric on $\mcstar$ and makes it a complete separable metric space (see e.g., \cite{bookBBI}). 

In the boundedly-compact case, the notion of \defstyle{Gromov-Hausdorff convergence} is also defined (see~\cite{Gr81} or~\cite{bookBBI}), which can be stated using~\eqref{eq:a_r} as follows. 
Let $\mathcal X_n=(X_n,o_n)$ be boundedly-compact \pms{} spaces ($n=1,2,\ldots$). The sequence $(\mathcal X_n)_n$ is said to converge to $\mathcal X=(X,o)$ in the Gromov-Hausdorff sense (Definition~8.1.1 of~\cite{bookBBI}) if for every $r>0$ and $0<\epsilon\leq r$, on has $\lim_n a(\epsilon,r;X_n,X) = 0$ (consider the zero measures in~\eqref{eq:a_r}). This defines a topology on $\mstar$. 

The metric $\cgh$ is identical to the restriction of the Gromov-Hausdorff-Prokhorov metric $\cghp$ to $\mcstar$ (by identifying $\mcstar$ with the subset $\{(X,o,\mu)\in\mmstar: \mu=0 \}$ of $\mmstar$). %
Now, define the \defstyle{Gromov-Hausdorff metric} $\gh$ on $\mstar$ to be the restriction of the metric $\ghp$~\eqref{eq:GHP} to $\mstar$. It can also be defined directly by~\eqref{eq:a_r} and~\eqref{eq:GHP} by letting the measures be the zero measures. 
Similarly to Theorem~\ref{thm:metrization}, we have

\begin{theorem}
	\label{thm:metrizationGH}
	The metric $\gh$ on $\mstar$, defined above,
	is a metrization of the Gromov-Hausdorff topology. Moreover, it makes $\mstar$ a complete separable metric space.
\end{theorem}

\begin{proof}
	The first claim is implied by Theorem~\ref{thm:convergence}.
	It can be seen that $\mstar$ is a closed subset of $\mmstar$. Therefore, Theorem~\ref{thm:GHPcomplete} implies that $\mstar$ is a complete separable metric space.
\end{proof}

In addition, a version of Theorem~\ref{thm:weak} holds for weak convergence of random boundedly-compact pointed metric spaces.

\subsection{Length Spaces}
\label{subsec:length}

In~\cite{AbDeHo13}, another version of the Gromov-Hausdorff-Prokhorov distance is defined in the case of \textit{length spaces}. It is shown below that it generates the same topology as (the restriction of) the metric $\ghp$.

A metric space $X$ is called a \defstyle{length space} if for all pairs $x,y\in X$, the distance of $x$ and $y$ is equal to the infimum length of the curves connecting $x$ to $y$. Let $\mathcal L$ be the set of (isometry classes of) pointed measured complete locally-compact length spaces (equipped with locally-finite Borel measures). For two elements $\mathcal X, \mathcal Y\in \mathcal L$, their distance is defined in~\cite{AbDeHo13} by the same formula as~\eqref{eq:lengthSpaceGHP}.
It is proved in~\cite{AbDeHo13} that this makes $\mathcal L$ a complete separable metric space.

Every element of $\mathcal L$ is boundedly-compact by Hopf-Rinow's theorem (see~\cite{AbDeHo13}). So $\mathcal L$ can be regarded as a subset of $\mmstar$. Now, consider the restriction of the metric $\ghp$ to $\mathcal L$. This metric is not equivalent to the metric in~\eqref{eq:lengthSpaceGHP}, but generates the same topology (by Theorem~\ref{thm:convergence}). Moreover, $\mathcal L$ is a closed subset of $\mmstar$ (see Theorem~8.1.9 of~\cite{bookBBI}). So Theorem~\ref{thm:GHPcomplete} implies that $\mathcal L$ is also complete and separable under the restriction of the metric $\ghp$.

In addition, the pre-compactness result Theorem~\ref{thm:GHP-precompactness} is a generalization of Theorem~2.11 of~\cite{AbDeHo13}.

\subsection{Random Measures}
\label{subsec:randomMeasure}
Let $S$ be a boundedly-compact metric space and $\mathcal M$ be the set of boundedly-finite Borel measures on $S$. 
The well known \textit{vague topology} on $\mathcal M$, makes it a Polish space
(see e.g., Lemma~4.6 in~\cite{bookKa17randommeasures}). 
This is the basis for having a standard probability space in defining \defstyle{random measures on $S$} as random elements in $\mathcal M$. The metrics defined in Remark~\ref{rem:noncompactHausdorff} are metrizations of the vague topology as well.

One can regard a random measure on $S$ as a random \pmm{} space by considering the natural map $\mu\mapsto (S,o,\mu)$ from $\mathcal M$ to $\mmstar$. The cost is considering measures on $S$ up to equivalence under automorphisms of $(S,o)$ (see also the next paragraph). 
This also allows the base space $(S,o)$ be random, and hence, 
a random \pmm{} space can also be called a \defstyle{random measure on a random environment}. %

To rule out the issue of the automorphisms in the above discussion, on can add marks to the points of $S$, which requires a generalization of the Gromov-Hausdorff-Prokhorov metric. See~\cite{Kh19generalization}.

\subsection{Benjamini-Schramm Metric For Graphs}
\label{subsec:graphs}

Benjamini and Schramm~\cite{BeSc01} defined a notion of convergence for rooted graphs, which is particularly interesting for studying the limit of a sequence of sparse graphs. 
For simple graphs, convergence under this metric is equivalent to the Gromov-Hausdorff convergence of the corresponding vertex sets equipped with the graph-distance metrics. %
Below, it is shown that, roughly speaking, the boundedly-compact case of the Gromov-Hausdorff metric defined in this paper generalizes the Benjamini-Schramm metric for simple graphs.  %
So random rooted graphs  can be regarded as random pointed metric  spaces. 

For simplicity, we restrict attention to simple graphs.	 It is also assumed that the graph is connected and \textit{locally-finite}; i.e., every vertex has finite degree. For two rooted networks $(G_1,o_1)$ and $(G_2,o_2)$, their distance is defined by $1/(\alpha+1)$, where $\alpha$ is the supremum of those $r>0$ such that there is a graph-isomorphism between $\cball{r}{o_1}$ and $\cball{r}{o_2}$ that maps $o_1$ to $o_2$. Let $\mathcal G_*$ be the set of isomorphism-classes of rooted graphs. It is claimed in~\cite{processes} that this distance function makes $\mathcal G_*$ a complete separable metric space.

Since we assume the graphs are simple, every graph $G$ can be modeled as a metric space, where the metric (which is the graph-distance metric) is integer-valued. Also, being locally-finite implies that the metric space is boundedly-compact.
So $\mathcal G_*$ can be identified with a subset of $\mstar$. It can be seen that the restriction of the Gromov-Hausdorff metric on $\mstar$ (defined in Subsection~\ref{subsec:GH}) to $\mathcal G_*$ is  equivalent to the metric defined in~\cite{processes} mentioned above.

\subsection{Discrete Spaces}
\label{subsec:discrete}

Let $\mathcal D_*$ be the set of all pointed discrete metric spaces (up to pointed isometries) which are boundedly-finite; i.e., every closed ball contains finitely many points. To study random pointed discrete spaces, \cite{dimensionI} defines a metric on $\mathcal D_*$ and shows that $\mathcal D_*$ is a Borel subset of some complete separable metric space. 
It is shown below that random pointed discrete spaces are special cases of random \pmm{} spaces (or random \pms{} spaces). 

First, $\mathcal D_*$ is clearly a subset of $\mstar$. Therefore, the generalization of the Gromov-Hausdorff metric on $\mstar$ (introduced in Subsection~\ref{subsec:GH}) induces a metric on $\mathcal D_*$ (the topology of this metric is discussed below). It should be noted that $\mathcal D_*$ is not a closed subset of $\mstar$, and hence, is not complete (in fact, $\mathcal D_*$ is dense in $\mstar$). However, it is a Borel subset of $\mstar$. %

Second, by equipping every discrete set $X$ with the counting measure on $X$, $\mathcal D_*$ can be regarded as a subset of $\mmstar$. It can be seen that it is a Borel subset which is not closed (e.g., $\{0,\frac{1}{n}\}$ converges to a single point whose measure is 2). The closure of $\mathcal D_*$ in $\mmstar$ is the set of 
elements of $\mmstar$ in which the underlying metric space is discrete and the measure is integer-valued and has full support.

By Theorem~\ref{thm:convergence}, it can be seen that the topology on $\mathcal D_*$ induced from $\mmstar$ coincides with the topology defined in~\cite{dimensionI}. However, it is strictly finer than the topology induced from $\mstar$. Nevertheless, it can be seen that these topologies induce the same Borel sigma-field on $\mathcal D_*$.

\subsection{The Gromov-Hausdorff-Vague Topology}

In~\cite{AtLoWi16}, a variant of the GHP metric is defined on the set $\mmstar'$ of boundedly-compact \textit{metric measure spaces} and its Polishness is proved. This space is slightly different from $\mmstar$ since in the former, the features outside the support of the underlying measure are discarded (see~\cite{bookVi10} for more discussion on the two different viewpoints). More precisely, two pointed metric measure spaces $(X,o_X,\mu_X)$ and $(Y,o_Y,\mu_Y)$ are called equivalent in~\cite{AtLoWi16} if there exists a measure preserving isometry between $\supp (\mu_X)\cup\{o_X\}$ and $\supp(\mu_Y)\cup\{o_Y\}$ that maps $o_X$ to $o_Y$. %
The set $\mmstar'$ can be mapped naturally into $\mmstar$ (by replacing $X$ with $\supp(\mu_X)\cup\{o_X\}$). The image of this map is the set of $(X,o,\mu)$ in $\mmstar$ such that $\supp(\mu)\supseteq X\setminus\{o\}$. Since the image of this map is not closed in $\mmstar$, the set $\mmstar'$ is not complete under the metric induced by the GHP metric (this holds even in the compact case). In~\cite{AtLoWi16}, another metric is defined that  makes $\mmstar'$ complete and separable. %
It can be seen that it generates the same topology as the restriction of the GHP metric to $\mmstar'$. %
A second proof for Polishness of $\mmstar'$ can be given by Alexandrov's theorem by using Polishness of $\mmstar$ and by showing that $\mmstar'$ corresponds to a $G_{\delta}$ subspace of $\mmstar$ (given $n>0$, it can be shown that the set of $(X,o,\mu)\in \mmstar$ such that $\forall x\in X:\mu(B_{1/n}(x))>0$ is open).

The method of~\cite{AtLoWi16} is different from the present paper. It defines the metric on $\mmstar'$ by modifying \eqref{eq:lengthSpaceGHP} (since \eqref{eq:lengthSpaceGHP} does not make $\mmstar'$ complete), but the definition in the present paper is based on the notion of \pmm{}-subspaces, Lemma~\ref{lem:GHP-monotone} and~\eqref{eq:GHP}. As mentioned in Remark~\ref{rem:variant}, this method gives more quantitative bounds in the arguments. Despite some similarities in the arguments (which are also similar to those of~\cite{AbDeHo13} and other literature that use the localization method to generalize the Gromov-Hausdorff metric), the results of~\cite{AtLoWi16} do not give a metrization of the Gromov-Hausdorff-Prokhorov topology on $\mmstar$ and do not imply its Polishness. 
Also, the Strassen-type theorems (Theorems~\ref{thm:strassen} and~\ref{thm:GHPcompact}) and the results based on them are new in the present paper.

The term \textit{Gromov-Hausdorff-vague topology} is used in~\cite{AtLoWi16} to distinguish it with another notion called \textit{the Gromov-Hausdorff-weak topology} defined therein. By considering only probability measures in the above discussion, the two topologies on the corresponding subset of $\mmstar'$ will be identical.

\subsection{The Skorokhod Space of C\`adl\`ag Functions}
\label{subsec:skorokhod}

The Skorokhod space, recalled below, is the space of c\`adl\`ag functions with values in a given metric space. By noting that every boundedly-compact \pmm{} space can be represented as a c\`adl\`ag curve in $\mmcstar$ (see the following lemma), one can consider the Skorokhod metric on $\mmstar$. This subsection studies the relations of this metric with the metric $\ghp$. %
By similar arguments, one can also study the connections of the Skorokhod space to the boundedly-compact cases in \cite{AbDeHo13}, \cite{processes}, \cite{AtLoWi16}, \cite{bookBBI} and \cite{bookVi10}, which are introduced earlier in this section.
\begin{lemma}
	\label{lem:cadlag2}
	For every boundedly-compact \pmm{} space $\mathcal X$, the curve $t\mapsto \pcball{\mathcal X}{t}$ is a c\`adl\`ag function with values in $\mmcstar$. Moreover, the left limit of this curve at $t=r$ is $(\overline{\oball{r}{o}},o,\restrict{\mu}{\oball{r}{o}})$, where $\overline{\oball{r}{o}}$ is the closure of $\oball{r}{o}$.
\end{lemma}
\begin{proof}
	The claim follows from Lemma~\ref{lem:cadlag} and~\eqref{eq:hp-ineq}.
\end{proof}

Let $S$ be a complete separable metric space. The \defstyle{Skorokhod space} $\mathcal D(S)$ is the space of all c\`adl\`ag functions $f:[0,\infty)\to S$. %
In~\cite{bookBi99}, a metric is defined on $\mathcal D(S)$ which is called the \defstyle{Skorokhod metric} here. %
Heuristically, two c\`adl\`ag functions $x_1,x_2\in\mathcal D(S)$ are close if by restricting $x_1$ to a large interval $[0,M]$ and by perturbing the time a little (i.e., by composing $x_1$ with a function which is close to the identity function), the resulting function is close in the sup metric to the restriction of $x_2$ to a large interval. The precise definition is skipped for brevity (see Section~16 of~\cite{bookBi99}). Under this metric, $\mathcal D(S)$ is a complete separable metric space.

Now let $S:=\mmcstar$. 
For every boundedly-compact \pmm{} space $\mathcal X$, let $\rho(\mathcal X)$ denote the curve $t\mapsto \pcball{\mathcal X}{t}$ with values in $\mmcstar$. By Lemma~\ref{lem:cadlag2}, the latter is c\`adl\`ag; i.e., is an element of $\mathcal D(S)$. Now, $\rho$ defines a function from $\mmstar$ to $\mathcal D(S)$. It can be seen that $\rho$ is injective and its image is
\[
\left\{x\in \mathcal D(S): \forall r\leq s: \pcball{x(s)}{r}=x(r) \right\}.
\]
It can also be seen that the latter is a closed subset of $\mathcal D(S)$. Therefore, the Skorokhod metric can be pulled back by $\rho$ to make $\mmstar$ a complete separable metric space.

\begin{proposition}
	\label{prop:skorokhod}
	One has
	\begin{enumerate}[(i)]
		\item \label{prop:skorokhod:top} The topology on $\mmstar$ induced by the Skorokhod metric (defined above) is strictly finer than the Gromov-Hausdorff-Prokhorov topology.
		\item \label{prop:skorokhod:borel} The Borel sigma-field of the Skorokhod metric on $\mmstar$ is identical with that of the Gromov-Hausdorff-Prokhorov metric.
	\end{enumerate}
\end{proposition}
\begin{proof}
	\eqref{prop:skorokhod:top}. 
	First, assume $\mathcal X_n\to \mathcal X$ in the Skorokhod topology. Theorem~16.2 of~\cite{bookBi99} implies that $\pcball{\mathcal X}{r}_n\to\pcball{\mathcal X}{r}$ for every continuity radius $r$ of $\mathcal X$. So Theorem~\ref{thm:convergence} gives that $\mathcal X_n\to\mathcal X$ under the metric $\ghp$.
	
	Second, let $\mathcal X_n:=\{0, 1+\frac 1n, -1-\frac 2n\}$ and $X:=\{0,1,-1\}$ equipped with the Euclidean metric and the zero measure (or the counting measure) and pointed at 0. Then $\mathcal X_n\to \mathcal X$ in the metric $\ghp$ but the convergence does not hold in the Skorokhod topology (note that for $r_n=1+\frac 1 n$, the ball $\pcball{\mathcal X}{r_n}_n$ is close to $\{0,1\}$, but is not close to any ball in $\mathcal X$ centered at 0).

	\eqref{prop:skorokhod:borel}. It can be seen that the set of c\`adl\`ag step functions with finitely many jumps is dense in $\mathcal D(S)$. Also, it can be seen that the the set $I$ of $\mathcal X\in\mmstar$ such that $\rho(\mathcal X)$ is such a curve (equivalently, the set $\{\pcball{\mathcal X}{r}: r\geq 0 \}$ is finite) is dense in $\mmstar$ under the Skorokhod topology. It can be seen that the sets $A^{N}_{\epsilon}(\mathcal X)$ for $\mathcal X\in I$ and $\epsilon>0$ generate the Skorokhod topology on $\mmstar$, where $A^N_{\epsilon}(\mathcal X)$ is defined as follows: If $r_1,r_2,\ldots,r_k$ are the set of discontinuity points of $\mathcal X$, $r_0:=0$ and $r_k<N$, consider the set of $\mathcal Y\in\mmstar$ such that there exists $0=:t_0<t_1<\cdots<t_k<t_{k+1}:=N$ such that for all $i\leq k$, one has $\norm{t_i-r_i}<\epsilon$ and for all $t_i\leq t<t_{i+1}$, one has $\cghp(\pcball{\mathcal Y}{t}, \pcball{\mathcal X}{r_i})<\epsilon$. It is left to the reader to show that this is a Borel subset of $\mmstar$ under the metric $\ghp$. This proves the claim.
\end{proof}

\begin{remark}
	If $\mathcal X_n\to \mathcal X$ under the metric $\ghp$ and $\mathcal X$ has no discontinuity radii, then the convergence holds in the Skorokhod topology as well. This follows from the fact that the curves $\rho(\mathcal X_n)$ converge to $\rho(\mathcal X)$ uniformly  on bounded intervals, which follows from Theorem~\ref{thm:convergence} and Lemma~\ref{lem:GHP-monotone}.
\end{remark}

\begin{remark}
	The above result means that to consider $\mmstar$ as a standard probability space, one could consider the Skorokhod metric on $\mmstar$ from the begging. This method is identical to considering the GHP metric if one is interested only in the Borel structure. However, the topology and the notion of weak convergence are different under these metrics. Nevertheless, in most of the examples in the literature that study scaling limits (e.g., the Brownian continuum random tree of~\cite{Al91-crtI}), both notions of convergence hold since the limiting spaces under study usually have no discontinuity radii.
\end{remark}

%

\bibliography{bib} 
\bibliographystyle{plain}

\end{document}